\documentclass[a4paper,12pt]{amsart}
\usepackage{latexsym, amsfonts, amsmath, amsthm, amssymb, verbatim, enumerate}
\usepackage{tikz}
\usepackage[top=3cm, bottom=3cm, left=3.2cm, right=3.2cm]{geometry}
\newcommand{\AC}{AC}
\newcommand{\BV}{{BV}}

\newcommand{\CTPP}{\hbox{$CT\kern-0.2ex{P}\kern-0.2ex{P}$}}
\newcommand{\Pol}{\mathcal{P}}

\newcommand{\mN}{\mathbb{N}}
\newcommand{\mR}{\mathbb{R}}
\newcommand{\mC}{\mathbb{C}}

\newcommand{\abs}[1]{\left\lvert#1\right\rvert}

\newcommand{\veca}{{\boldsymbol{a}}}
\newcommand{\vecb}{{\boldsymbol{b}}}
\newcommand{\vecc}{{\boldsymbol{c}}}
\newcommand{\vecm}{{\boldsymbol{m}}}
\newcommand{\vecp}{{\boldsymbol{p}}}
\newcommand{\vecx}{{\boldsymbol{x}}}
\newcommand{\vecy}{{\boldsymbol{y}}}
\newcommand{\vecu}{{\boldsymbol{u}}}
\newcommand{\vecv}{{\boldsymbol{v}}}
\newcommand{\vecw}{{\boldsymbol{w}}}
\newcommand{\vecz}{{\boldsymbol{z}}}

\newcommand{\norm}[1]{\left\lVert#1\right\rVert}

\newcommand{\normbv}[1]{\left\lVert#1\right\rVert_{\BV(\sigma)}}

\newcommand\st{\thinspace : \thinspace}
\def\ls[#1,#2]{\overline{\vphantom{\vbox to 1.2 ex{}} #1\, #2}}
\def\lso[#1,#2]{\overline{\vphantom{\vbox to 1.2 ex{}} #1\, #2}^\circ}

\def\snorm#1{\Bigl \Vert #1 \Bigr \Vert}
\def\ssnorm#1{\Vert #1 \Vert}
\def\sparen(#1){\Bigl ( #1 \Bigr )}
\def\ssparen(#1){ (#1) }
\newcommand\plist[1]{\bigl[ #1 \bigr]}

\newcommand{\interior}[1]{\mathop{\mathrm{int}}(#1)}

\DeclareMathOperator{\var}{var}
\DeclareMathOperator{\cvar}{\rm cvar}
\DeclareMathOperator{\vf}{vf}

 \newtheorem{thm}{Theorem}[section]
 \newtheorem{cor}[thm]{Corollary}
 \newtheorem{lem}[thm]{Lemma}
 
 \theoremstyle{definition}
 \newtheorem{defn}[thm]{Definition}
 \theoremstyle{remark}
 \newtheorem{rem}[thm]{Remark}
 \newtheorem{ex}[thm]{Example}
 \numberwithin{equation}{section}

\newcommand{\booktitle}[1]{\textit{#1}}
\newcommand{\journalname}[1]{\textrm{#1}}

\begin{document}

%opening
\title{Isomorphisms of $AC(\sigma)$ spaces}
\author{Ian Doust and Michael Leinert}

\address{Ian Doust, School of Mathematics and Statistics,
University of New South Wales,
UNSW Sydney 2052 Australia}%
\email{i.doust@unsw.edu.au}

\address{Michael Leinert, Institut f{\"u}r Angewandte Mathematik,
Universit{\"a}t Heidelberg,
Im Neuenheimer Feld 294,
D-69120 Heidelberg
Germany}
\email{leinert@math.uni-heidelberg.de}
\date{28 November 2013}

\begin{abstract}
Analogues of the classical Banach-Stone theorem for spaces of continuous functions are studied in the context of the spaces of absolutely continuous functions introduced by Ashton and Doust. We show that if $\AC(\sigma_1)$ is algebra isomorphic to $\AC(\sigma_2)$ then $\sigma_1$ is homeomorphic to $\sigma_2$. The converse however is false. In a positive direction we show that the converse implication does hold if the sets $\sigma_1$ and $\sigma_2$ are confined to a restricted collection of compact sets, such as the set of all simple polygons.
\end{abstract}

\maketitle

\section{Introduction}

In \cite{AD1} Ashton and Doust defined the Banach algebra $AC(\sigma)$ consisting of `absolutely continuous' functions with domain an arbitrary nonempty compact subset $\sigma$ of $\mC$ (or equivalently of $\mR^2$). The motivation for their definition was to extend the spectral theory of well-bounded operators to cover operators whose spectra need not be contained in the real line. This led to the definition of an $AC(\sigma)$ operator being a bounded operator on a Banach space $X$ which admits a bounded functional calculus $\Psi: \AC(\sigma) \to B(X)$.  Under some additional assumptions, the image of this map $\Psi$ is an algebra of operators which is isomorphic to  $\AC(\sigma)$.
Accordingly, one can recover certain aspects of the theories of normal operators and of scalar-type spectral operators, replacing algebras of continuous functions $C(\Omega)$ with algebras of absolutely continuous functions.
Quite naturally then, underlying many of the open problems in this area are questions which ask for analogues of the classical topological results about $C(\Omega)$ spaces. (Details of the theory of $AC(\sigma)$ operators can be found in \cite{AD3}.)

One of the most classical of these topological results is the Banach-Stone theorem which says that two compact Hausdorff spaces $\Omega_1$ and $\Omega_2$ are homeomorphic if and only if the function algebras $C(\Omega_1)$ and $C(\Omega_2)$ are linearly isometric. There have been many generalizations and extensions of this result (see \cite{GJ}).
Work of Amir \cite{Am} shows that one may still deduce that $\Omega_1$ and $\Omega_2$ are homeomorphic if one only assumes that $C(\Omega_1)$ and $C(\Omega_2)$ are (linearly) isomorphic with Banach-Mazur distance less than $2$.  Cohen \cite{Co} has shown that the value $2$ is sharp.

In a different direction, one  might require that the spaces $C(\Omega_1)$ and $C(\Omega_2)$ be isomorphic as algebras. In this case one may argue using the maximal ideal spaces to get the same conclusion. In particular, as the next result shows, if $C(\Omega_1)$ and $C(\Omega_2)$ are algebra isomorphic, then they are isometrically isomorphic. This result was originally proved in \cite{GK}; a modern treatment is given in \cite{GJ}.

\begin{thm}[Gelfand and Kolmogoroff 1939] Let $\Omega_1$ and $\Omega_2$ be compact
Hausdorff spaces. Then  $C(\Omega_1)$ and $C(\Omega_2)$ are isomorphic as algebras if  and only if  $\Omega_1$
and $\Omega_2$ are homeomorphic. Moreover, every algebra isomorphism $j : C(\Omega_1) \to
C(\Omega_2)$ is of the form $j(f) = f \circ h$ where $h : \Omega_1 \to \Omega_2$ is a homeomorphism.
\end{thm}

The main issue that we shall address in this paper is the corresponding relationship between the topological structure of the set $\sigma$ and the algebraic structure of $\AC(\sigma)$. In Section~\ref{prelims} we shall recall the definition and main properties of $AC(\sigma)$ and then give a simple proof that if $\AC(\sigma_1)$ and $\AC(\sigma_2)$ are algebra isomorphic, then $\sigma_1$ and $\sigma_2$ are homeomorphic. The converse of this is false however. In Section~\ref{disk-square} we show that the algebra of absolutely continuous functions over the closed unit disk is not isomorphic to the algebra of absolutely continuous functions over a square.

If, however, one restricts the class of sets in which $\sigma$ may lie, one can recover some sort of analogue of the Banach-Stone Theorem. In Theorem~\ref{polygons} we show that if $P_1$ and $P_2$ are simple polygons, then $\AC(P_1)$ is algebra isomorphic to $\AC(P_2)$. In Section~\ref{poly-with-holes} we extend this result to cover more general sets based on polygons.

We note that different applications have led to quite a number of different concepts of absolute continuity for functions of two or more variables. The reader is cautioned that these concepts are generally distinct, and often, as is the case here, impose particular conditions on the domains of the functions considered. The definition of absolute continuity that is studied here was developed to have specific properties which are appropriate for the intended application to spectral theory, namely:
\begin{enumerate}
 \item it should apply to functions defined on the spectrum of a bounded operator, that is, an arbitrary nonempty compact subset $\sigma$ of the plane,
 \item it should agree with the usual definition if $\sigma$ is an interval in $\mR$;
 \item $\AC(\sigma)$ should contain all sufficiently well-behaved functions;
 \item if $\alpha,\beta \in \mC$ with $\alpha \ne 0$, then the space $\AC(\alpha \sigma + \beta)$ should be isometrically isomorphic to $\AC(\sigma)$.
\end{enumerate}
The interested reader may consult \cite{AD2}, \cite{mC} and \cite{dB} for a sample of what is known about the relationships between some of these definitions.

\smallskip
\textbf{Notation.} Suppose that $\mathcal{A}$ and $\mathcal{B}$ are Banach algebras.
We shall write $\mathcal{A} \simeq \mathcal{B}$ to mean that $\mathcal{A}$ and $\mathcal{B}$ are isomorphic (as Banach algebras).
%, and write $\mathcal{A} \cong \mathcal{B}$ to mean that $\mathcal{A}$ and $\mathcal{B}$ are isometrically isomorphic.

\smallskip

Throughout, we shall use the term polygon to refer to a simple polygon including its interior. In particular, every such polygon is homeomorphic to the closed unit disk.

\section{Preliminaries}\label{prelims}

In this section we shall briefly outline the definition of the spaces $\AC(\sigma)$. Here we follow \cite{DL1} rather than the original definitions given in \cite{AD1}. Throughout, $\sigma$, $\sigma_1$ and $\sigma_2$ will denote nonempty compact subsets of the plane.
Although the original motivation for these definitions came from considering functions defined on subsets of the complex plane, for this paper it will be notationally easier consider the domains of the functions to be subsets of $\mR^2$.
We shall work throughout with algebras of complex-valued functions.

Suppose that $f: \sigma \to \mC$. Let $S =
\plist{\vecx_0,\vecx_1,\dots,\vecx_n}$ be a finite ordered list of elements of $\sigma$, where, for the moment, we shall assume that $n \ge 1$.
Note that the elements of such a list do not need to be distinct.

We define the \textit{curve variation of $f$ on the set $S$} to be
\begin{equation} \label{lbl:298}
    \cvar(f, S) =  \sum_{i=1}^{n} \abs{f(\vecx_{i}) - f(\vecx_{i-1})}.
\end{equation}

We shall also need to measure the `variation factor' of the list $S$. Loosely speaking, this is the greatest number of times that $\gamma_S$ crosses any line in the plane, where
$\gamma_S$ denotes the piecewise linear curve joining the points of $S$ in order. The following definition makes precise just what is meant by a crossing.

\begin{defn}\label{crossing-defn}
Suppose that $\ell$ is a line in the plane. We say that $\ls[\vecx_i,\vecx_{i+1}]$, the line segment joining $\vecx_i$ to $\vecx_{i+1}$, is a \textit{crossing segment} of $S = \plist{\vecx_0,\vecx_1,\dots,\vecx_n}$ on $\ell$ if any one of the following holds:
\begin{enumerate}[(i)]
  \item $\vecx_i$ and $\vecx_{i+1}$ lie on (strictly) opposite sides of $\ell$.
  \item $i=0$ and $\vecx_i \in \ell$.
  \item $i > 0$, $\vecx_i \in \ell$ and $\vecx_{i-1} \not\in \ell$.
  \item $i=n-1$, $\vecx_i \not\in \ell$ and $\vecx_{i+1} \in \ell$.
\end{enumerate}
In this case we shall write $\ls[\vecx_i,\vecx_{i+1}] \in X(S,\ell)$.
\end{defn}

\begin{defn}\label{vf-defn}
Let $\vf(S,\ell)$ denote the number of crossing segments of $S$ on $\ell$. The \textit{variation factor} of $S$ is defined to be
 \[ \vf(S) = \max_{\ell} \vf(S,\ell). \]
\end{defn}

Clearly $1 \le \vf(S) \le n$. For completeness, in the case that
$S = \plist{\vecx_0}$ we set $\cvar(f, \plist{\vecx_0}) = 0$ and let $\vf(\plist{\vecx_0},\ell) = 1$ whenever $\vecx_0 \in \ell$.

\begin{ex}
Consider the line $\ell$ and the list $S = [\vecx_i]_{i=0}^8$ as shown in Figure~\ref{crossings}. Let $s_i = \ls[\vecx_i,\vecx_{i+1}]$. Then the crossing segments for $S$ on $\ell$ are $s_0$ (rule (ii)), $s_2$ (rule (i)), $s_4$ (rule (iii)) and $s_7$ (rule (iv)). All the other segments are not crossing segments of $S$ on $\ell$. Thus $\vf(S,\ell) = 4$.

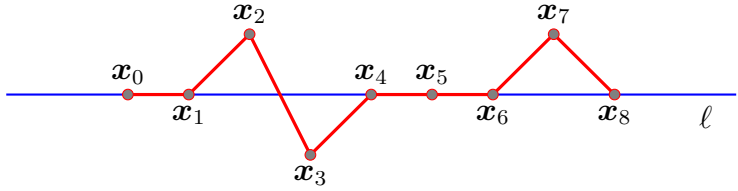
\begin{figure}[!ht]
\begin{center}
\begin{tikzpicture}[scale=0.8]
%
% Left picture
%
 \draw[thick,blue] (-1,0) --(11,0);
 \draw[very thick, red] (1,0) -- (2,0) -- (3,1) -- (4,-1) -- (5,0) -- (6,0) -- (7,0) -- (8,1) -- (9,0);
\draw[red] (1,0) node[circle, draw, fill=black!50,inner sep=0pt, minimum width=4pt] {};
\draw (1,0) node[above] {$\vecx_0$};
\draw[red] (2,0) node[circle, draw, fill=black!50,inner sep=0pt, minimum width=4pt] {};
\draw (2,0) node[below] {$\vecx_1$};
\draw[red] (3,1) node[circle, draw, fill=black!50,inner sep=0pt, minimum width=4pt] {};
\draw (3,1) node[above] {$\vecx_2$};
\draw[red] (4,-1) node[circle, draw, fill=black!50,inner sep=0pt, minimum width=4pt] {};
\draw (4,-1) node[below] {$\vecx_3$};
\draw[red] (5,0) node[circle, draw, fill=black!50,inner sep=0pt, minimum width=4pt] {};
\draw (5,0) node[above] {$\vecx_4$};
\draw[red] (6,0) node[circle, draw, fill=black!50,inner sep=0pt, minimum width=4pt] {};
\draw (6,0) node[above] {$\vecx_5$};
\draw[red] (7,0) node[circle, draw, fill=black!50,inner sep=0pt, minimum width=4pt] {};
\draw (7,0) node[below] {$\vecx_6$};
\draw[red] (8,1) node[circle, draw, fill=black!50,inner sep=0pt, minimum width=4pt] {};
\draw (8,1) node[above] {$\vecx_7$};
\draw[red] (9,0) node[circle, draw, fill=black!50,inner sep=0pt, minimum width=4pt] {};
\draw (9,0) node[below] {$\vecx_8$};

 \draw (10.5,0) node[below] {$\ell$};

\end{tikzpicture}
\caption{Examples of crossing segments.}\label{crossings}
\end{center}
\end{figure}

\end{ex}

The \textit{two-dimensional variation} of a function $f : \sigma
\rightarrow \mC$ is defined to be
\begin{equation} \label{lbl:994}
    \var(f, \sigma) = \sup_{S}
        \frac{ \cvar(f, S)}{\vf(S)},
\end{equation}
where the supremum is taken over all finite ordered lists of elements of $\sigma$.
The \textit{variation norm} is
  \[ \normbv{f} = \norm{f}_\infty + \var(f,\sigma) \]
and the set of functions of bounded variation on $\sigma$ is
  \[ \BV(\sigma) = \{ f: \sigma \to \mC \st \normbv{f} < \infty\}. \]
The space $\BV(\sigma)$ is a Banach algebra under pointwise operations \cite[Theorem 3.8]{AD1}. If $\sigma = [0,1]$ then the above definition is equivalent to the more classical one.

Let $\Pol_2$ denote the space of polynomials in two real variables of the form $p(x,y) = \sum_{n,m} c_{nm} x^n y^m$, and let $\Pol_2(\sigma)$ denote the restrictions of elements on $\Pol_2$ to $\sigma$. The algebra $\Pol_2(\sigma)$ is always a subalgebra of $\BV(\sigma)$ \cite[Corollary 3.14]{AD1}.

\begin{defn}
The set of absolutely continuous functions on $\sigma$, denoted $\AC(\sigma)$, is the closure of $\Pol_2(\sigma)$ in $\BV(\sigma)$.
\end{defn}

The set $\AC(\sigma)$ forms a closed subalgebra of $\BV(\sigma)$ and hence is a Banach algebra.

We shall say that $f \in C^1(\sigma)$ if there exists an open neighbourhood $U$ of $\sigma$ and
an extension $F$ of $f$ to $U$ such that the partial derivatives of $F$ (of order one) are continuous on $U$.
The space $\CTPP(\sigma)$ consists of those functions $f$ for which there is a triangulation of a neighbourhood $U$ of $\sigma$ and an extension of $f$ to $U$ which is continuous and piecewise planar on this triangulation.
It was shown in \cite{DL1} that both $C^1(\sigma)$ and $\CTPP(\sigma)$ are dense subsets of $\AC(\sigma)$.

Our first step is to show that if $\AC(\sigma_1)$ and $\AC(\sigma_2)$ are isomorphic as algebras, then $\sigma_1$ and $\sigma_2$ must be homeomorphic. We note that one does not need to assume that the isomorphism is continuous.

\begin{lem}\label{spec_rad}
 Suppose that $f \in \AC(\sigma)$. Then the spectrum of $f$ is $\sigma(f) = f(\sigma)$ and hence the spectral radius of $f$ is $r(f) = \norm{f}_\infty$.
\end{lem}

\begin{proof}
 This is more or less immediate from \cite[Corollary 3.9]{AD1}.
\end{proof}

\begin{thm}\label{alg_props}
 Suppose that $j: \AC(\sigma_1) \to \AC(\sigma_2)$ is an algebra isomorphism. Then
\begin{enumerate}
 \item\label{pt1}
        $\norm{f}_\infty =  \norm{j(f)}_\infty$ for all $f \in \AC(\sigma_1)$.
 \item there exists a homeomorphism $h: \sigma_1 \to \sigma_2$.
 \item\label{pt3}
         $j(f) = f \circ h^{-1}$ for all $f \in \AC(\sigma_1)$.
 \item $j$ is continuous.
\end{enumerate}
\end{thm}

\begin{proof}
Since $j$ preserves the identity element, it also preserves the spectrum of elements. Thus, using Lemma~\ref{spec_rad},
  \[ \norm{f}_\infty = r(f) = r(j(f)) = \norm{j(f)}_\infty
  \]
for all $f \in \AC(\sigma)$. Since $\AC(\sigma_1)$ is dense in $C(\sigma_1)$, this implies that $j$ extends to an isometric isomorphism $\hat \jmath : C(\sigma_1) \to C(\sigma_2)$ and hence, by the Banach-Stone Theorem, $\sigma_1$ is homeomorphic to $\sigma_2$. Indeed there exists a homeomorphism $h: \sigma_1 \to \sigma_2$ such that  ${\hat \jmath}(f) = f \circ h^{-1}$ for all $f \in C(\sigma_1)$. Restricting this to $\AC(\sigma_1)$ gives part~\ref{pt3}.

Suppose that $f_n \to f$ in $\AC(\sigma_1)$. Then certainly $f_n \to f$ uniformly and hence pointwise.
Suppose that $j(f_n) \to g$ in $\AC(\sigma_2)$ (and hence also pointwise). Then for all $x \in \sigma_2$,
  \[ g(x) = \lim_n j(f_n)(x) = \lim_n f_n(h^{-1}(x)) = f(h^{-1}(x)) = j(f)(x) \]
and hence $j(f) = g$. Thus, by the Closed Graph Theorem, $j$ is continuous.
\end{proof}

It is easy to find homeomorphic sets $\sigma_1$ and $\sigma_2$ for which $\AC(\sigma_1)$ and $\AC(\sigma_2)$ are algebra isomorphic, but not isometrically. If the isomorphism preserves norms, then part~\ref{pt1} of the above theorem implies that it also preserves variation.

\begin{cor}\label{pres_var} Suppose that $j: \AC(\sigma_1) \to \AC(\sigma_2)$ is an isometric Banach algebra isomorphism. Then
 $\var(f,\sigma_1) = \var(j(f),\sigma_2)$ for all $f \in \AC(\sigma_1)$.
\end{cor}

\begin{ex}
 Let $\sigma_1 = \{0,1,2\}$ and $\sigma_2 = \{0,1,i\}$. Since $\sigma_1 \subseteq \mR$,
  \[ \norm{f}_{\BV(\sigma_1)} = \norm{f}_\infty + |f(1)-f(0)|+|f(2)-f(1)|, \qquad f \in \AC(\sigma_1). \]
On the other hand, any function defined on $\sigma_2$ can clearly be written in the form
  $f(x+iy) = g(y-x)$ for some function $g$ of one real variable. Lemma~3.12 and Proposition~3.10 of \cite{AD1} then imply that the norm for  $f \in  \AC(\sigma_2)$ is given by
  \[ \norm{f}_{\BV(\sigma_2)} = \norm{f}_\infty +
      \max\bigl(|f(1)-f(0)|, |f(i)-f(0)|, |f(i)-f(1)|\bigr).
  \]
% One can check, for example, that
% %(in the real scalar case),
% the unit balls of these two three-dimensional algebras have different numbers of faces and hence these algebras can not be isometrically isomorphic.
Any isomorphism must map idempotents to idempotents. However it is easy to see that while all idempotents in $\AC(\sigma_2)$ have variation at most 1, the algebra $\AC(\sigma_1)$ contains the idempotent with $f(0) = f(2) = 0$ and $f(1) = 1$ which has variation 2. Thus these algebras can not be isometrically isomorphic.
\end{ex}

% A more sophisticated example is given later.

In the other direction, suppose that $\alpha: \mR^2 \to \mR^2$ is an invertible affine transformation. 
It is clear from the definition of variation that $\norm{f}_{\BV(\sigma)} = \norm{f \circ \alpha^{-1}}_{\BV(\alpha(\sigma))}$. Since affine transformations preserve polynomials, it is clear that $\AC(\sigma)$ is isometrically isomorphic to $\AC(\alpha(\sigma))$. (This is a very small extension of \cite[Theorem 4.1]{AD1}.) In Section~\ref{HPAM} we shall extend this to slightly more general transformations of the plane, at the expense of the algebra isomorphism no longer being isometric.

\section{The disk and the square}\label{disk-square}

%\section{Random topological lemmas}\label{rand-top}

Let $Q = [0,1] \times [0,1] \subseteq \mR^2$ denote the closed unit square and let $D = \{\vecx \in \mR^2 \st \norm{\vecx} \le 1\}$ denote the closed unit disk. These sets are clearly homeomorphic.
The aim of this section is to show that $\AC(Q) \not\simeq \AC(D)$.

\begin{thm}\label{square-circle} $\AC(Q)$ and $\AC(D)$ are not isomorphic as algebras.
%, and in particular, $\AC(Q) \not \simeq \AC(D)$.
\end{thm}

\begin{proof}
 Suppose that $j: \AC(Q) \to \AC(D)$ is an algebra isomorphism. By Theorem~\ref{alg_props}, the map $j$ is continuous and hence $\norm{j(f)}_{\AC(D)} \le \norm{j} \norm{f}_{\AC(Q)}$ for all $f \in \AC(Q)$.

Let $h: Q \to D$ be the homeomorphism associated with $j$.  Then $h([0,1]\times \{0\})$ is a closed arc on the unit circle $\partial D$.

Let $n \in \mN$ be even. % and let $x_k = \frac{k}{n}$ for $0 \le k \le n$.
For $0 \le k \le n$, let $\vecp_k = h(\frac{k}{n},0)$. Now choose $\epsilon_n > 0$ small enough so that, for every odd $k$, the $\epsilon_n$-disc with centre $\vecp_k$ does not meet the line segment $\ls[\vecp_{k-1},\vecp_{k+1}]$.

Now let $\delta_n > 0$ be chosen (using the uniform continuity of $h$) so that if $\vecx,\vecx' \in Q$ with $\norm{\vecx-\vecx'} \le \delta_n$ then $\norm{h(\vecx)-h(\vecx')} < \epsilon_n$.  Without loss we may assume that  $\delta_n \le 1$.

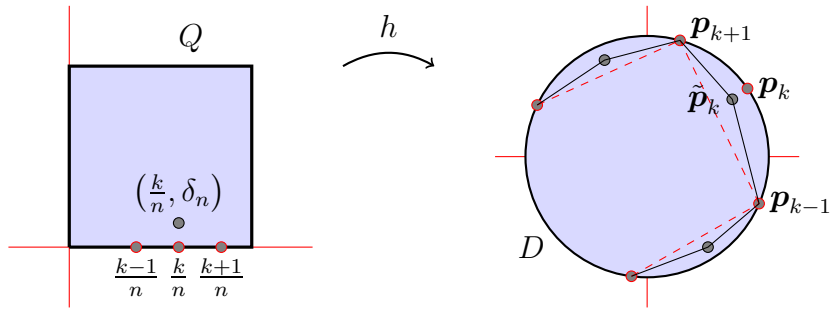
\begin{figure}[!ht]
\begin{center}
\begin{tikzpicture}[scale=0.8]
%
% Left picture
%
 \draw[red] (-1,0) -- (4,0) ;   % axes
 \draw[red] (0,-1) -- (0,4);
 \fill[blue!15] (0,0) -- (3,0) -- (3,3) -- (0,3) -- (0,0);   % set Q
 \draw[very thick,black] (0,0) -- (3,0) -- (3,3) -- (0,3) -- (0,0);

 \draw[red] (1.1,0) node[circle, draw, fill=black!50,inner sep=0pt, minimum width=4pt] {};
 \draw[red] (1.8,0) node[circle, draw, fill=black!50,inner sep=0pt, minimum width=4pt] {};
 \draw[red] (2.5,0) node[circle, draw, fill=black!50,inner sep=0pt, minimum width=4pt] {};
% \draw[red] (1.1,0.3) node[circle, draw, fill=black!50,inner sep=0pt, minimum width=4pt] {};
 \draw[black] (1.8,0.4) node[circle, draw, fill=black!50,inner sep=0pt, minimum width=4pt] {};
% \draw[red] (2.5,0.3) node[circle, draw, fill=black!50,inner sep=0pt, minimum width=4pt] {};

 \draw[black] (1.1,0) node[below] {$\frac{k-1}{n}$};
 \draw[black] (1.8,0) node[below] {$\frac{k}{n}$};
 \draw[black] (2.5,0) node[below] {$\frac{k+1}{n}$};
 \draw[black] (1.8,0.4) node[above] {$\left(\frac{k}{n},\delta_n\right)$};

 \draw[black] (2,3.05) node[above] {$Q$};

  \path[thick,->] (4.5,3) edge [bend left] (6,3);
  \draw[black] (5.25,3.3) node[above] {$h$};

 \draw[red] (7,1.5) -- (12,1.5) ;
 \draw[red] (9.5,-1) -- (9.5,4);

 \fill[blue!15] (9.5,1.5) circle (2cm);

 \draw[thick,black] (9.5,1.5) circle (2cm);

 \draw[red] (11.34,0.7212) node[circle, draw, fill=black!50,inner sep=0pt, minimum width=4pt] {};
 \draw[red] (11.15,2.629) node[circle, draw, fill=black!50,inner sep=0pt, minimum width=4pt] {};
 \draw[red] (10.04, 3.427) node[circle, draw, fill=black!50,inner sep=0pt, minimum width=4pt] {};
 \draw[red,dashed] (11.34,0.7212) -- (10.04, 3.427);
 \draw[black] (11.34,0.7212) node[right] {$\vecp_{k-1}$};
 \draw[black] (11.15,2.629) node[right] {$\vecp_{k}$};
 \draw[black] (10.04, 3.6) node[right] {$\vecp_{k+1}$};
% \draw[red] (11.46,1.103) node[circle, draw, fill=black!50,inner sep=0pt, minimum width=4pt] {};
% \draw[red] (1.8,0.3) node[circle, draw, fill=black!50,inner sep=0pt, minimum width=4pt] {};
% \draw[red] (2.5,0.3) node[circle, draw, fill=black!50,inner sep=0pt, minimum width=4pt] {};
 \draw[black] (10.9,2.45) node[circle, draw, fill=black!50,inner sep=0pt, minimum width=4pt] {};
 \draw[black] (10.9,2.45) node[left] {$\tilde{\vecp}_{k}$};

 \draw[red] (9.242, -0.483) node[circle, draw, fill=black!50,inner sep=0pt, minimum width=4pt] {};
 \draw[red,dashed] (9.242, -0.483) -- (11.34,0.7212);
 \draw[red] (7.692, 2.355) node[circle, draw, fill=black!50,inner sep=0pt, minimum width=4pt] {};
 \draw[red,dashed] (7.692, 2.355) -- (10.04, 3.427);

 \draw[black] (10.5,0) node[circle, draw, fill=black!50,inner sep=0pt, minimum width=4pt] {};
 \draw[black] (8.8, 3.1) node[circle, draw, fill=black!50,inner sep=0pt, minimum width=4pt] {};
 \draw[black] (9.242, -0.483) -- (10.5,0) -- (11.34,0.7212) -- (10.9,2.45) -- (10.04, 3.427) --(8.8, 3.1) -- (7.692, 2.355);

 \draw[black] (8,0) node[left] {$D$};

\end{tikzpicture}
\caption{The construction of $\tilde{\vecp}_k$ in the proof of Theorem~\ref{square-circle}.}
\end{center}
\end{figure}

For each odd $k$, let $\tilde{\vecp}_k = h(\frac{k}{n},\delta_n) \in B(\vecp_k,\epsilon_n)$.
Let $S = [\vecp_0,\tilde{\vecp}_1,\vecp_2,\tilde{\vecp}_3,\dots,\linebreak[1] \tilde{\vecp}_{n-1},\vecp_n] \subseteq D$.
 It is easy to see that the points of $S$ form the vertices of a convex subset of $D$ and so, in particular, $\vf(S_n) = 2$.

Consider the  map $f_n: Q \to \mR$ defined by $f_n(x,y) = \min(y/\delta_n,1)$. Clearly $f_n \in \AC(Q)$ with $\norm{f_n}_{\AC(Q)} =2$. Define $g_n: D \to \mR$ by $g_n = f_n \circ h^{-1} = j(f_n)$.
Then $g_n(\vecp_k) = 0$ for $k$ even, and $g_n(\tilde{\vecp}_k) = 1$ for $k$ odd. Thus
  \[ \norm{g_n}_{\AC(D)} \ge  \var(g_n,D) \ge \frac{\cvar(g_n,S_n)}{\vf(S_n)} = \frac{n}{2}. \]
But for all $n$, $\norm{g_n}_{\AC(D)} \le \norm{j} \norm{f_n}_{\AC(Q)} \le 2 \norm{j}$ and hence we have a contradiction.
\end{proof}

% If we insist that our homeomorphism is sufficiently smooth, then the only algebra isomorphisms are those induced by affine maps. As we shall see, this fails if you weaken this smoothness assumption too much.

As we shall now show, there are severe restrictions on the behaviour of any algebra isomorphism  which is associated with a $C^2$ homeomorphism from $Q$ to another compact subset of the plane. 

\begin{lem}
 Suppose that $\Omega \subseteq \mR^2$ is compact. Then a set $\ell \subseteq \Omega$ is a closed line segment if and only if it is closed, convex and can be disconnected by the removal of a single point.
\end{lem}

\begin{proof} The forward implication is obvious.

Suppose now that the second condition holds. Let $\vecx \in \ell$ denote a point whose removal splits $\ell \setminus \{\vecx\}$ into disjoint sets $\ell_1$ and $\ell_2$. Choose points $\vecy_1 \in \ell_1$ and $\vecy_2 \in \ell_2$. Then $\ls[\vecy_1, \vecy_2]$ lies inside $\ell$. and must pass through $\vecx$ as $\ell_1 \cup \ell_2$ is not connected. Since $y_2$ was an arbitrary element of $\ell_2$, the line through $\vecy_1$ and $\vecx$ contains every element of $\ell_2$ --- and similarly every element of $\ell_1$ must lie on the same line. Thus $\ell$ is a closed convex subset of a line, or in other words, a line segment.
\end{proof}

Recall that a set $U$ is mid-point convex if $\frac{1}{2}(\vecx + \vecy) \in U$ for all $\vecx,
\vecy \in U$.

\begin{lem}
 Suppose that $\Omega \subseteq \mR^2$ is compact and that $\ell = \{\vecx + \lambda \vecv \st 0 \le \lambda \le 1\} \subseteq \Omega$ is a closed line segment. Let $h: \Omega \to \sigma \subseteq \mR^2$ be a homeomorphism. Then $h(\ell)$ is a line segment if and only if it is mid-point convex.
\end{lem}

\begin{proof}
Again the forward implication is clear.

Now suppose that $h(\ell)$ is mid-point convex. Since $h$ is a homeomorphism, $h(\ell)$ is closed and can be disconnected by a point. Since $h(\ell)$ is closed and mid-point convex, it is convex, and so the result follows from the previous lemma.
\end{proof}

\begin{lem}
 Suppose that $\sigma \subseteq \mR^2$  and that $h: Q \to \sigma$ is a homeomorphism. For $y \in [0,1]$ let $\ell_y = [0,1] \times \{y\}$. If $h(\ell_0)$ is not a line segment, then there exists $\delta > 0$ such that $h(\ell_y)$ is not a line segment for any $y \in [0,\delta]$.
\end{lem}

\begin{proof}
 As $h(\ell_0)$ is not a line segment we may choose $x,x' \in [0,1]$ such that $\vecv = \frac{1}{2}(h(x,0)+h(x',0))$ is not an element of $h(\ell_0)$. Let $\epsilon = d(\vecv,h(\ell_0)) > 0$. Now choose $\delta > 0  $ small enough so that if $\vecu,\vecu' \in Q$ with $\norm{\vecu-\vecu'} \le \delta$ then $\norm{h(\vecu)-h(\vecu')} < \epsilon/3$.

Suppose that $0 \le y \le \delta$ and that $h(\ell_y)$ is a line segment. Since $h(\ell_y)$ is mid-point convex, there exists $t \in [0,1]$ such that $h(t,y) = \frac{1}{2}(h(x,y)+h(x',y))$. But then
\begin{align*}
 \norm{h(t,0) - \vecv} & \le \snorm{h(t,0)-h(t,y)}
               + \snorm{h(t,y) - \frac{h(x,y)+h(x',y)}{2}} \\
          & \qquad\qquad\qquad\qquad
               + \snorm{\frac{h(x,y)+h(x',y)}{2} - \frac{h(x,0)+h(x',0)}{2} } \\
   & \le \frac{\epsilon}{3} + 0 + \frac{\epsilon}{3}   < \epsilon
\end{align*}
contradicting that $d(\vecv,h(\ell_0)) = \epsilon$.
\end{proof}

A consequence of this result is that if $h:Q \to \sigma$ is a homeomorphism and there exists a line segment $\ell \in Q$ such that $h(\ell)$ is not a line segment, then we may assume that both $\ell$ and $h(\ell)$ lie in the interiors of their respective sets.

\begin{thm}\label{C2-homs}
 Suppose that $\emptyset \ne \sigma \subseteq \mR^2$ is compact and that $j: \AC(Q) \to \AC(\sigma)$ is an algebra isomorphism with associated homeomorphism $h: Q \to \sigma$. If $h$ is $C^2$ then $h$ maps line segments to line segments.
\end{thm}

\begin{proof}
Suppose that there exist $\sigma$, $j$ and $h$ as above such that for some line segment $\ell \subseteq Q$, $h(\ell)$ is not a line segment. We shall show that this leads to a contradiction. In order to streamline the proof, a number of simplifications can be made.

By the above remark we can assume that $\ell$ lies in the interior of $Q$. Since $h$ is $C^2$, $h(\ell)$ has a tangent at each point and the curve can at least locally be considered as the graph of a $C^2$ function of a parametrization of this tangent line. 
Since $h(\ell)$ is not a line segment, one can therefore choose an invertible affine map $\beta: \mR^2 \to \mR^2$ such that there is a subsegment $\ell_0$ of $\ell$ such that 
\begin{enumerate}
 \item $(\beta \circ h)(\ell_0) = \{(s,t(s)) \,:\, 0 \le s \le 1\}$, and 
 \item $t''(s) > 0$ for $0 < s < 1$.
\end{enumerate}

Now choose an invertible affine map $\alpha: \mR^2 \to \mR^2$ such that 
\begin{enumerate}
 \item $\alpha(\ell_0) = [0,1] \times \{0\} \subset [0,1] \times [-1,1] \subseteq \interior{\alpha(Q)}$, and
 \item $(\beta \circ h \circ \alpha^{-1})([0,1] \times [0,1])$ lies above $(\beta \circ h)(\ell_0)$.
\end{enumerate}
We shall write $h_1 = \beta \circ h \circ \alpha^{-1}$ for the homeomorphism from $\alpha(Q)$ to $\beta(\sigma)$, and $C$ for the curve $(\beta \circ h)(\ell_0)$. (See Figure~\ref{C2-homs-pic}.)

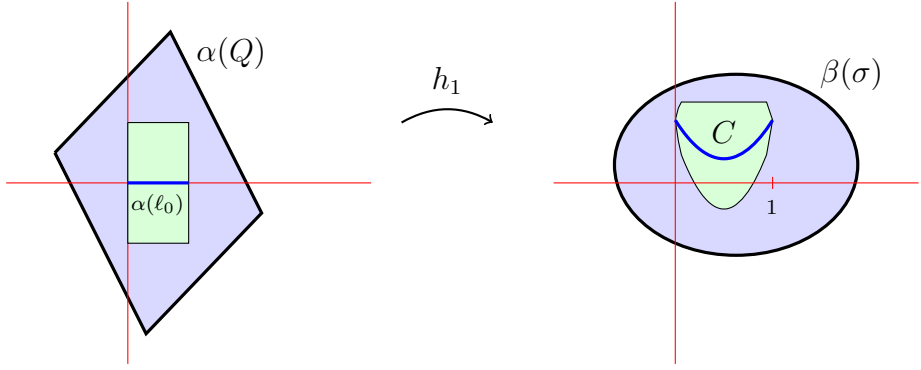
\begin{figure}[!ht]
\begin{center}
\begin{tikzpicture}[scale=0.8,domain=9:10]
%
% Left picture
%

 \fill[blue!15] (-1.2,0.5) -- (0.7,2.5) -- (2.2,-0.5) -- (0.3,-2.5) -- (-1.2,0.5);   % set \alpha(Q)
 \draw[very thick,black] (-1.2,0.5) -- (0.7,2.5) -- (2.2,-0.5) -- (0.3,-2.5) -- (-1.2,0.5); 
 \draw[red] (-2,0) -- (4,0) ;   % axes
 \draw[red] (0,-3) -- (0,3);
 \fill[green!15] (0,1) -- (1,1) -- (1,-1) -- (0,-1) -- (0,1);
 \draw[black] (0,1) -- (1,1) -- (1,-1) -- (0,-1) -- (0,1);
 \draw[very thick, blue] (0,0) -- (1,0);

%  \draw[black] (1.1,0) node[below] {$\frac{k-1}{n}$};
%  \draw[black] (1.8,0) node[below] {$\frac{k}{n}$};
%  \draw[black] (2.5,0) node[below] {$\frac{k+1}{n}$};
%  \draw[black] (1.8,0.4) node[above] {$\left(\frac{k}{n},\delta_n\right)$};

 \draw[black] (1.7,1.7) node[above] {$\alpha(Q)$};
 \draw[black] (0.5,0) node[below] {${}_{\alpha(\ell_0)}$};

  \path[thick,->] (4.5,1) edge [bend left] (6,1);
  \draw[black] (5.25,1.3) node[above] {$h_1$};

 \fill[blue!15] (10,0.3) ellipse (2.0 and 1.5);
 \draw[very thick,black] (10,0.3) ellipse (2.0 and 1.5);  % \beta(\sigma)
 \draw[black] (11.2,1.8) node[right] {$\beta(\sigma)$};

 \fill[green!15] (9.00,1.04) -- (9.05,1.24) -- (9.10,1.34) -- (9.15,1.34) -- (9.20,1.34) -- (9.25,1.34) -- (9.30,1.34) -- (9.35,1.34) -- (9.40,1.34) -- (9.45,1.34) -- (9.50,1.34) -- (9.55,1.34) -- (9.60,1.34) -- (9.65,1.34) -- (9.70,1.34) -- (9.75,1.34) -- (9.80,1.34) -- (9.85,1.34) -- (9.90,1.34) -- (9.95,1.34) -- (10.00,1.34) -- (10.00,1.34) -- (10.10,1.34) -- (10.20,1.34) -- (10.20,1.34) -- (10.20,1.34) -- (10.30,1.34) -- (10.40,1.34) -- (10.40,1.34) -- (10.40,1.34) -- (10.50,1.34) -- (10.60,1.04) -- (10.60,1.04) -- (10.60,1.04) -- (10.50,0.465) -- (10.40,0.227) -- (10.40,0.227) -- (10.40,0.227) -- (10.30,0.025) -- (10.20,-0.14) -- (10.20,-0.14) -- (10.20,-0.14) -- (10.10,-0.268) -- (10.00,-0.36) -- (10.00,-0.36) -- (9.95,-0.393) -- (9.90,-0.415) -- (9.85,-0.43) -- (9.80,-0.433) -- (9.75,-0.43) -- (9.70,-0.415) -- (9.65,-0.393) -- (9.60,-0.36) -- (9.55,-0.32) -- (9.50,-0.268) -- (9.45,-0.21) -- (9.40,-0.14) -- (9.35,-0.063) -- (9.30,0.025) -- (9.25,0.12) -- (9.20,0.227) -- (9.15,0.34) -- (9.10,0.465) -- (
9.05,0.697) -- (9.00,1.04);

 \draw[black] (9.00,1.04) -- (9.05,1.24) -- (9.10,1.34) -- (9.15,1.34) -- (9.20,1.34) -- (9.25,1.34) -- (9.30,1.34) -- (9.35,1.34) -- (9.40,1.34) -- (9.45,1.34) -- (9.50,1.34) -- (9.55,1.34) -- (9.60,1.34) -- (9.65,1.34) -- (9.70,1.34) -- (9.75,1.34) -- (9.80,1.34) -- (9.85,1.34) -- (9.90,1.34) -- (9.95,1.34) -- (10.00,1.34) -- (10.00,1.34) -- (10.10,1.34) -- (10.20,1.34) -- (10.20,1.34) -- (10.20,1.34) -- (10.30,1.34) -- (10.40,1.34) -- (10.40,1.34) -- (10.40,1.34) -- (10.50,1.34) -- (10.60,1.04) -- (10.60,1.04) -- (10.60,1.04) -- (10.50,0.465) -- (10.40,0.227) -- (10.40,0.227) -- (10.40,0.227) -- (10.30,0.025) -- (10.20,-0.14) -- (10.20,-0.14) -- (10.20,-0.14) -- (10.10,-0.268) -- (10.00,-0.36) -- (10.00,-0.36) -- (9.95,-0.393) -- (9.90,-0.415) -- (9.85,-0.43) -- (9.80,-0.433) -- (9.75,-0.43) -- (9.70,-0.415) -- (9.65,-0.393) -- (9.60,-0.36) -- (9.55,-0.32) -- (9.50,-0.268) -- (9.45,-0.21) -- (9.40,-0.14) -- (9.35,-0.063) -- (9.30,0.025) -- (9.25,0.12) -- (9.20,0.227) -- (9.15,0.34) -- (9.10,0.465) -- (9.
05,0.697) -- (9.00,1.04);

 \draw[red] (7,0) -- (13,0) ;
 \draw[red] (9,-3) -- (9,3);   % axes
 \draw[red] (10.6,-0.1) -- (10.6,0.1);
 \draw[black] (10.6,-0.1) node[below] {${}_{1}$};
 \draw[black] (9.8,1.2) node[below] {$C$};

 \draw[very thick, blue] (9.00,1.04) -- (9.05,0.962) -- (9.10,0.89) -- (9.15,0.822) -- (9.20,0.76) -- (9.25,0.702) -- (9.30,0.65) -- (9.35,0.602) -- (9.40,0.56) -- (9.45,0.522) -- (9.50,0.49) -- (9.55,0.462) -- (9.60,0.44) -- (9.65,0.422) -- (9.70,0.41) -- (9.75,0.402) -- (9.80,0.40) -- (9.85,0.402) -- (9.90,0.41) -- (9.95,0.422) -- (10.00,0.44) -- (10.00,0.44) -- (10.10,0.49) -- (10.20,0.56) -- (10.20,0.56) -- (10.20,0.56) -- (10.30,0.65) -- (10.40,0.76) -- (10.40,0.76) -- (10.40,0.76) -- (10.50,0.89) -- (10.60,1.04) -- (10.60,1.04);

\end{tikzpicture}
\caption{The homeomorphism $h_1: \alpha(Q) \to \beta(\sigma)$ in the proof of Theorem~\ref{C2-homs}.}
\end{center}\label{C2-homs-pic}
\end{figure}

The proof now mimics that of Theorem~\ref{square-circle}. Let $n \in \mN$ be even. For $0 \le k \le n$, let $\vecp_k = (\frac{k}{n},t(\frac{k}{n}))$ and let $x_k \in [0,1]$ denote the (unique) number such that $h_1(x_k,0) = \vecp_k$.

Choose $\epsilon_n$ small enough so that for all odd $k$, the ball $B(\vecp_k,\epsilon_n)$ lies beneath the  line segment $\ls[\vecp_{k-1},\vecp_{k+1}]$. (This is of course possible by the convexity of the function $t$.)

Now choose $0 < \delta_n < 1$ so that if $\vecx,\vecy \in \alpha(Q)$ with $\norm{\vecx-\vecx'} \le \delta_n$ then $\norm{h_1(\vecx)-h_1(\vecx')} < \epsilon_n$.

For each odd $k$, let $\tilde{\vecp}_k = h_1(x_k,\delta) \in B(\vecp_k,\epsilon)$. Then $\tilde{\vecp}_k$ lies above the curve $C$ but below the chord $\ls[\vecp_{k-1},\vecp_{k+1}]$.
Let $S_n = [\vecp_0,\tilde{\vecp}_1,\vecp_2,\tilde{\vecp}_3,\dots,\linebreak[1]\tilde{\vecp}_{n-1},\vecp_n]$, so that as in the proof of Theorem~\ref{square-circle}, $\vf(S_n) = 2$.

Consider the  map $f_n: \alpha(Q) \to \mR$ defined by
  \[ f_n(x,y) = \begin{cases}
                1,  &\hbox{if $y \ge \delta_n$}, \\
                y/\delta_n,  & \hbox{if $0 \le y < \delta_n$}, \\
                0,  &\hbox{if $y < 0$.}
                 \end{cases}
  \]
Clearly $f_n \in \AC(\alpha(Q))$ with $\norm{f_n}_{\AC(\alpha(Q))} =2$. Again define $g_n: \beta(\sigma) \to \mR$ by $g_n = f_n \circ h_1^{-1}$ to produce a function with $\norm{g_n}_{\AC(\beta(\sigma))} \ge n/2$.
But, noting the remarks at the end of Section~\ref{prelims} about the invariance of variation norms under affine transformations, we have that for all $n$, $\norm{g_n}_{\AC(\beta(\sigma))} \le \norm{j} \norm{f_n}_{\AC(\alpha(Q))} \le 2 \norm{j}$ which is the required contradiction.
\end{proof}

Functions mapping line segments to line segments have been studied by various authors. We refer the reader to \cite{CCDEM} for further details. It is worth noting that such maps need not be affine. For example $h(x,y) = ((x+1)/(y+1),2y/(y+1))$, which maps the unit square to the trapezoid with vertices $(1,0),(2,0),(1,1)$ and $(\frac{1}{2},1)$, is a nonaffine line-segment preserving function.

\section{Half-plane-affine maps}\label{HPAM}

Despite the results of the previous section there are some positive statements that can be made about when $\AC(\sigma_1)$ and $\AC(\sigma_2)$ are isomorphic. As we noted earlier, this is certainly the case if $\sigma_2$ is the image of $\sigma_1$ under an affine homeomorphism $\alpha$. In this section we weaken this condition on the homeomorphism mapping $\sigma_1$ to $\sigma_2$.

\begin{defn} A \textit{half-plane splitting} of $\mR^2$ is a pair of closed half-planes $\{H_1,H_2\}$ whose union is $\mR^2$ and which only intersect along their shared boundary line.
\end{defn}

\begin{defn} An invertible map $\alpha: \mR^2 \to \mR^2$ is said to be a \textit{half-plane-affine map} if there exists a half-plane splitting $\{H_1,H_2\}$ of $\mR^2$ and two affine maps $\alpha_1,\alpha_2:\mR^2 \to \mR^2$ such that $\alpha(\vecx) = \alpha_j(\vecx)$ whenever $\vecx \in H_j$. We shall write $\alpha = \{\alpha_1,\alpha_2\}_{H_1,H_2}$.
\end{defn}

Any half-plane-affine map is clearly continuous. The assumption of invertibility ensures that $\{\alpha(H_1),\alpha(H_2)\}$ is a half-plane splitting of $\mR^2$. We also have the following easy fact.

% The condition clearly implies continuity. Note that half-plane-affine maps do not send line segments to line segments, but also are not $C^2$ homeomorphisms of the plane.

\begin{lem}\label{hpa-inverse} The inverse of a half-plane-affine map is  a half-plane-affine map.
\end{lem}

Suppose for the remainder of this section that $\alpha= \{\alpha_1,\alpha_2\}_{H_1,H_2}$ is a half-plane-affine map.

We shall show below that for such maps $\BV(\sigma) \simeq \BV(\alpha(\sigma))$. The main point in proving this is showing that given any finite ordered list $S$ of elements of $\sigma$, $\vf(S)$ is comparable to $\vf(\alpha(S))$. Heuristically, if the number of times that the curve $\gamma_S$ crosses  a line $\ell$ is $k$, then $\gamma_{\alpha(S)}$ should cross either $\alpha_1(\ell)$ or $\alpha_2(\ell)$ at least $\frac{k}{2}$ times. Proving this is  a little delicate however because in the case that a segment $s_i = \ls[\vecx_i,\vecx_{i+1}]$  has at least one of its endpoints on the line $\ell$, it is possible that $s_i$ is a crossing segment on $\ell$ but that $\ls[\alpha(\vecx_i),\alpha(\vecx_{i+1})]$ is not a crossing segment on either $\alpha_1(\ell)$ or $\alpha_2(\ell)$.

\begin{lem}\label{hpa-vf}
Let $\alpha: \mR^2 \to \mR^2$ be a half-plane affine map. Suppose that $S = [\vecx_0,\dots,\vecx_n]$ is a finite ordered list of points in $\mR^2$ and that ${\hat S} = [\vecv_0,\dots,\vecv_n]$ is the list of the images of these points under $\alpha$. Then
  \[ \frac{1}{2} \vf(S) \le \vf({\hat S}) \le 2 \vf(S). \]
\end{lem}

\begin{proof} By the previous lemma it suffices to just prove the left-hand inequality.

Let $k = \vf(S)$ and fix a line $\ell$ such that $\vf(S,\ell) = k$.
For $0 \le i \le n-1$ let $s_i = \ls[\vecx_i,\vecx_{i+1}]$ and let ${\hat s}_i = \ls[\vecv_i,\vecv_{i+1}]$. Our aim is to find a correspondence between crossing segments of $S$ on $\ell$ and crossing segments of ${\hat S}$ on either $\alpha_1(\ell)$ or $\alpha_2(\ell)$. 

If ${\hat \ell} = \alpha(\ell)$ is also a line then $s_i \in X(S,\ell)$ if and only if ${\hat s}_i \in X({\hat S},{\hat \ell})$ and so $\vf(\hat S) \ge \vf(S)$ which certainly gives the required inequality.
This occurs in particular if $\ell$ is parallel to the boundary line between $H_1$ and $H_2$.

If ${\hat \ell}$ is not a line, then $\alpha_1(\ell)$ and $\alpha_2(\ell)$ do not coincide, and 
there must exist a unique point $\vecw \in \ell$ that lies on the shared boundary of $H_1$ and $H_2$. 

%In this case $\alpha_1(\ell)$ and $\alpha_2(\ell)$ do not coincide.
% Let $\vecw$ denote the point on $\ell$ that lies on the shared boundary of $H_1$ and $H_2$.
% 
% Suppose first that $\alpha_1(\ell)$ and $\alpha_2(\ell)$ coincide. In this case $s_i \in X(S,\ell)$ if and only if ${\hat s}_i \in X({\hat S},{\hat \ell})$ where ${\hat \ell} = \alpha(\ell)$. Thus $\vf(\hat S) \ge \vf(S)$ which certainly gives the required inequality.
% 
% We shall assume then that $\alpha_1(\ell)$ and $\alpha_2(\ell)$ do not coincide. (Note that this implies that $\ell$ is not the boundary line between $H_1$ and $H_2$.)
Suppose that $s_i  \in X(S,\ell)$. If $\vecx_i$ and $\vecx_{i+1}$ lie strictly on opposite sides of $\ell$, then $\vecv_i$ and $\vecv_{i+1}$ lie strictly on opposite sides of $\hat \ell$, and so ${\hat s}_i$ is a crossing segment for $\hat S$ for at least one of $\alpha_1(\ell)$ or $\alpha_2(\ell)$. The more difficult situation is if one of the endpoints of $s_i$ lies on $\ell$. Referring to Definition~\ref{crossing-defn} we have the following possibilities:
\begin{enumerate}[(i)]
 \item $i=0$ and $\vecx_i \in \ell$. Then $\vecv_0 \in {\hat \ell}$ and hence ${\hat s}_0$ is a crossing segment of $\hat S$ on at least one of $\alpha_1(\ell)$ or $\alpha_2(\ell)$.
 \item\label{case ii}
$i > 0$, $\vecx_i \in \ell$ and $\vecx_{i-1} \not\in \ell$. Note that in this case $s_{i-1} \not\in X(S,\ell)$.
% Suppose first that $\vecx_i \in H_1$.
% If $\vecv_{i-1} \not\in \alpha_1(\ell)$, then ${\hat s}_i$ is a crossing segment of $\hat S$ on $\alpha_1(\ell)$.
% If $\vecv_{i-1} \in \alpha_1(\ell)$, then (as $\alpha_1(\ell)$ and $\alpha_2(\ell)$ do not coincide) $\vecx_{i-1} \in H_2$ and $\vecv_{i-1} \not\in \alpha_2(\ell)$. In this case
% ${\hat s}_i \not\in X({\hat S},\alpha_1(\ell))$. If $\vecx_i = \vecw$ then ${\hat s}_i \in X({\hat S},\alpha_2(\ell))$.

Without loss of generality we may label the half-planes so that $\vecx_i \in H_1$. Now ${\hat s}_i$ is a crossing segment of $\hat S$ on $\alpha_1(\ell)$ except in the case that $\vecv_{i-1} \in \alpha_1(\ell)$ (see Figure~\ref{hpam-vf-proof}). If $\vecv_{i-1} \in \alpha_1(\ell)$ then, as $\alpha_1(\ell)$ and $\alpha_2(\ell)$ do not coincide, $\vecx_{i-1} \in H_2$ and $\vecv_{i-1} \not \in \alpha_2(\ell)$. 
We must now distinguish the case when $\vecx_i$ lies in the interior of $H_1$ and the case when $\vecx_i$ lies in the boundary of $H_1$.
% This implies that ${\hat s}_i \in X({\hat S},\alpha_2(\ell))$, although to see this one must distinguish the case when $\vecx_i$ lies in the interior of $H_1$ and the case when $\vecx_i$ lies in the boundary of $H_1$.

If $\vecx_i = \vecw$ then $\vecv_i \in \alpha_2(\ell)$ so since $\vecv_{i-1} \not \in \alpha_2(\ell)$ we have ${\hat s}_i \in X({\hat S},\alpha_2(\ell))$.
If $\vecx_i \ne \vecw$, then (as in Figure~\ref{hpam-vf-proof})
% \footnote{I don't think that this really needs to be expanded here, but it we did, I'd write: \dots then $\vecv_{i-1}$ and $\vecv_i$ lie in the interiors of the two half-planes, and so $\alpha(\vecw) \in \alpha_1(\ell)$ lies strictly between these two points. But $\alpha(\vecw)$ lies on $\alpha_2(\ell)$ and so $\vecv_{i-1}$ and $\vecv_i$ lie on opposite sides of $\alpha_2(\ell)$.}
 $\vecv_{i-1}$ and $\vecv_i$ lie on opposite sides of $\alpha_2(\ell)$ and hence ${\hat s}_{i-1} \in X({\hat S},\alpha_2(\ell))$. 

\begin{figure}[!ht]
\begin{center}
\begin{tikzpicture}[scale=0.8]
%
% Left picture
%
 \draw[thick,black] (2,0) -- (3,5);
 \draw[thick,red]  (0,4) -- (5,1);
 \draw[thick,blue] (1,3.4) -- (2,2.8) -- (4,3.2);
 \draw[blue] (1,3.4) node[circle, draw, fill=black!50,inner sep=0pt, minimum width=4pt] {};
 \draw[blue] (2,2.8) node[circle, draw, fill=black!50,inner sep=0pt, minimum width=4pt] {};
 \draw[blue] (4,3.2) node[circle, draw, fill=black!50,inner sep=0pt, minimum width=4pt] {};
% \draw[black] (2.5,2.5) node[circle, draw, fill=black!50,inner sep=0pt, minimum width=4pt] {};

 \draw (1,3.4) node[below] {$\vecx_{i+1}$};
 \draw (2,2.8) node[below] {$\vecx_i$};
 \draw (4,3.2) node[right] {$\vecx_{i-1}$};
% \draw (2.5,2.5) node[above] {$\vecw$};
 \draw (0.5,0) node[above] {$H_1$};
 \draw (3.5,0) node[above] {$H_2$};
 \draw (4.5,1.4) node[above] {$\ell$};
%
% arrow
%
  \path[thick,->] (6,4) edge [bend left] (7.5,4);
  \draw[black] (6.75,4.3) node[above] {$\alpha$};
%
% Right picture
%
 \draw[thick,black] (11,0) -- (12,5);
 \draw[thick,red]  (8,2.5) -- (11.5,2.5);
 \draw[thick,red] (11.5,2.5) -- (14,1);
 \draw[thick,red,dashed]  (11.5,2.5) -- (15,2.5);
 \draw[thick,red,dashed]  (9,4) -- (11.5,2.5);
 \draw[thick,blue] (9.5,2.5) -- (10.5,2.5) -- (13.5,2.5);
 \draw[blue] (9.5,2.5) node[circle, draw, fill=black!50,inner sep=0pt, minimum width=4pt] {};
 \draw[blue] (10.5,2.5) node[circle, draw, fill=black!50,inner sep=0pt, minimum width=4pt] {};
 \draw[blue] (13.5,2.5) node[circle, draw, fill=black!50,inner sep=0pt, minimum width=4pt] {};
 \draw (9.5,2.5) node[below] {$\vecv_{i+1}$};
 \draw (10.5,2.5) node[below] {$\vecv_i$};
 \draw (13.5,2.5) node[above] {$\vecv_{i-1}$};
 \draw (14,2) node[below] {$\alpha_2(\ell)$};
 \draw (8,2.5) node[above] {$\alpha_1(\ell)$};
 \draw (9.5,0) node[above] {$\alpha(H_1)$};
 \draw (12.5,0) node[above] {$\alpha(H_2)$};
\end{tikzpicture}
\caption{Mapping of crossing segments in the proof of Lemma~\ref{hpa-vf}.}\label{hpam-vf-proof}
\end{center}
\end{figure}
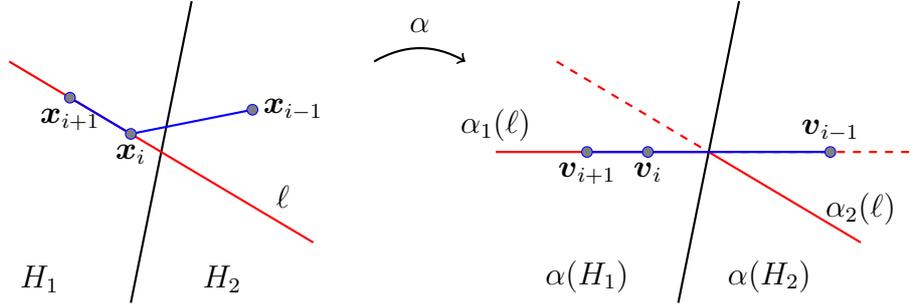

% The case that $\vecx_i \in H_2$ obviously follows by swapping the appropriate indices. In either case,
Thus, while only one of $s_{i-1}$ and $s_i$ are crossing segments of $S$ on $\ell$, at least one of ${\hat s}_{i-1}$ and ${\hat s}_i$ is a crossing segment of $\hat S$ on either $\alpha_1(\ell)$ or $\alpha_2(\ell)$.
  \item $i=n-1$, $\vecx_i \not\in \ell$ and $\vecx_{i+1} \in \ell$. Again we may assume that $\vecx_n \in H_1$ and hence that $\vecv_n \in \alpha_1(\ell)$.

If $\vecv_{n-1}  \not\in \alpha_1(\ell)$ then clearly ${\hat s}_{n-1} \in X({\hat S},\alpha_1(\ell))$. On the other hand, if $\vecv_{n-1} \in \alpha_1(\ell)$ then one may argue as in (\ref{case ii}) that ${\hat s}_{n-1} \in X({\hat S},\alpha_2(\ell))$.
%
% Then  either ${\hat s}_{n-1} \in X({\hat S},\alpha_1(\ell))$, or else, if $\vecv_{n-1}  \in \alpha_1(\ell)$, then ${\hat s}_{n-1} \in X({\hat S},\alpha_2(\ell))$. As we may again swap $H_1$ and $H_2$, we see that whichever half plane $\vecx_n$ lies in,
As before then, ${\hat s}_{n-1}$ is a crossing segment of $\hat S$ on at least one of $\alpha_1(\ell)$ or $\alpha_2(\ell)$.
\end{enumerate}
Let $k_1$ and $k_2$ be the number of crossing segments of $\hat S$ on %at least one of
$\alpha_1(\ell)$ and $\alpha_2(\ell)$ respectively. Then the above discussion shows that $k_1+k_2 \ge k$. It follows therefore that $\vf({\hat S}) \ge \frac{k}{2}$ as required.
\end{proof}

\begin{thm}\label{hpam-isoms} Let $\alpha: \mR^2 \to \mR^2$ be a half-plane affine map.
Suppose that $\sigma_1$ is a nonempty compact subset of $\mR^2$ and that $\sigma_2 = \alpha(\sigma_1)$. Then $\BV(\sigma_1) \simeq \BV(\sigma_2)$ and $\AC(\sigma_1) \simeq \AC(\sigma_2)$.
\end{thm}

\begin{proof} Suppose that $\alpha = \{\alpha_1,\alpha_2\}_{H_1,H_2}$ and that $\ell_0$ is the boundary line between $H_1$ and $H_2$.
For $f \in \BV(\sigma_1)$ let $\hat f: \sigma_2 \to \mR$ be defined by
  \[ \hat f(\alpha(\vecx)) = f(\vecx), \qquad \vecx \in \sigma_1.\]
The first step is to show that $\hat f \in \BV(\sigma_2)$.

By the previous lemma
  \[ \frac{\cvar(\hat f,{\hat S})}{\vf(\hat S)} \le 2 \frac{\cvar(f,S)}{\vf(S)} \le 2 \var(f,\sigma_1). \]
Taking the supremum over all such finite lists $\hat S$ shows that
  \[ \ssnorm{\hat f}_{\BV(\sigma_2)} \le 2 \norm{f}_{\BV(\sigma_1)} \]
and in particular that $\hat f \in \BV(\sigma_2)$. Let $j: \BV(\sigma_1) \to \BV(\sigma_2)$ be defined by $j(f) = \hat f$. It is clear that $j$ is a continuous algebra homomorphism. Lemma~\ref{hpa-inverse} can now be used to deduce that $j$ is also onto and hence that $j$ is a Banach algebra isomorphism.

Suppose that $g \in \CTPP(\sigma_1)$. Then $j(g)$ will also be planar on polygonal regions (on a neighbourhood) of $\sigma_2$. Indeed $j$ is a bijection from $CTPP(\sigma_1)$ to $CTPP(\sigma_2)$ and hence $j$ is also a bijection between the closures of these sets, $\AC(\sigma_1)$ and $\AC(\sigma_2)$.
\end{proof}

As the example below shows, the factor of $2$ in the above proof is necessary.

\begin{ex}
 Suppose that $\sigma_1$ is the nonconvex quadrilateral
%\footnote{The actual shape of this set isn't particularly important.}
% With a little more care one can probably show that you can do something similar with any nonconvex set.}
with vertices at $(1,0)$, $(0,4)$, $(-1,0)$ and $(0,2)$. Note that $\sigma_1$ is the image of the closed unit square under a half-plane-affine map.
Theorem~\ref{hpam-isoms} then  implies that $\AC(Q) \simeq \AC(\sigma_1)$.

%Now let $\sigma_2$ be any convex nonempty compact set in $\mR^2$.
Define
  \[ f(x,y) = \max(1-y,0), \qquad (x,y) \in \sigma_1. \]
Then, as $f$ only varies in the $y$ direction, it is clear that $\norm{f}_\infty = 1$, that $\var(f,\sigma_1) = 1$ and hence that $\norm{f}_{\BV(\sigma_1)} = 2$. Now write $f = f_1+f_2$ where $f_1(x,y) = f(x,y) \chi_{[-1,0]}(x)$ and $f_2(x,y) = f(x,y) \chi_{[0,1]}(x)$. Note  that both $f_1$ and $f_2$ are in $\CTPP(\sigma_1) \subseteq \AC(\sigma_1)$ and that these functions have disjoint supports.

Suppose now that $j: \AC(\sigma_1) \to \AC(Q)$ is a Banach algebra isomorphism, with associated homeomorphism $h: \sigma_1 \to Q$.
Let $g_1 = j(f_1)$, $g_2 = j(f_2)$ and $g = j(f)$. Using Theorem~\ref{alg_props}(\ref{pt1}), we can choose points $\vecz_1,\vecz_2 \in Q$ such that $g_k(\vecz_\ell) = \delta_{k\ell}$. Let $\gamma$ denote the line segment joining $\vecz_1$ and $\vecz_2$. Then $h^{-1}(\gamma)$ is a continuous curve in $\sigma_1$ which necessarily passes though some point $(0,y) \in \sigma_1$. Let $\vecw = h(0,y)$ be the corresponding point on $\gamma$. Then $g(\vecw) = g_1(\vecw)+g_2(\vecw) = f_1(0,y) + f_2(0,y) = 0$.
But this implies that $\cvar(g,\gamma) \ge 2$ and hence $\var(g,\sigma_2) \ge 2$. By Corollary~\ref{pres_var} we see that $j$ can not be isometric. In particular, this example shows that factor of 2 that appears in the proof of Theorem~\ref{hpam-isoms} is necessary.
\end{ex}

\begin{rem}\label{triangle-square}
A simple adjustment to the above example, replacing the vertex $(0,2)$ with the point $(0,0)$, shows that if $T$ is a closed triangular region in $\mR^2$, then $\AC(T) \simeq \AC(Q)$. Thus isomorphism class does not distinguish between the number of vertices in polygonal regions. We shall come back to this issue later in the paper.
\end{rem}

\section{Locally piecewise affine maps}

The results of the previous section can be extended to cover homeomorphisms of the plane made up from more than two affine maps.

Let $\alpha: \mR^2 \to \mR^2$ be an invertible affine map, and let $C$ be a convex $n$-gon. Then $\alpha(C)$ is also a convex $n$-gon. Denote the sides of $C$ by $s_1,\dots,s_n$. Suppose that $\vecx_0 \in \interior{C}$. The point $\vecx_0$ determines a triangulation $T_1,\dots,T_n$ of $C$, where $T_j$ is the (closed) triangle with side $s_j$ and vertex $\vecx_0$. A point $\vecy_0 \in \interior{\alpha(C)}$ determines a similar triangularization $\hat{T}_1,\dots,\hat{T}_n$ of $\alpha(C)$, where the numbering is such that $\alpha(s_j) \subseteq \hat{T}_j$. The following fact is then clear.

% Let $\alpha: \mR^2 \to \mR^2$ be an invertible affine map so that $\alpha(C)$ is also a convex $n$-gon. A point $\vecy_0 \in \interior{\alpha(C)}$ determines a similar triangularization $\hat{T}_1,\dots,\hat{T}_n$ of $\alpha(C)$, where the numbering is such that $\alpha(s_j) \subseteq \hat{T}_j$.

\begin{lem}\label{vertex-move} With the notation as above, there is a unique map $h: \mR^2 \to \mR^2$ such that
\begin{enumerate}
 \item $h(\vecx) = \alpha(\vecx)$ for $\vecx \not\in \interior{C}$,
 \item $h$ maps $T_j$ onto $\hat{T}_j$, for $1 \le j \le n$.
 \item $\alpha_j = h|T_j$ is affine, for $1 \le j \le n$.
 \item $h(\vecx_0) = \vecy_0$.
\end{enumerate}
\end{lem}

We shall say that $h$ is the \textit{locally piecewise affine} map determined by $(C,\alpha, \vecx_0,\vecy_0)$. 
% For notational simplicity, we shall let $T_0 = \mR^2 \setminus \interior{C}$ and let $\alpha_0 = h|T_0 = \alpha|T_0$.

%
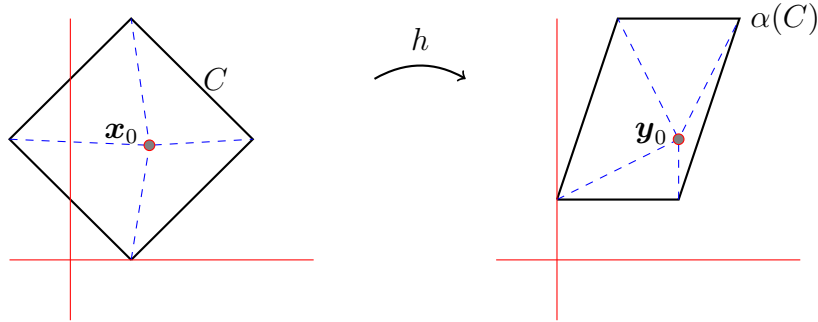
\begin{figure}[!ht]
\begin{center}
\begin{tikzpicture}[scale=0.8]
%
% Left picture
%
 \draw[thick,black] (1,0) -- (3,2) -- (1,4) -- (-1,2) -- (1,0);
 \draw[blue,dashed]  (1,0) -- (1.3,1.9) -- (1,4);
 \draw[blue,dashed]  (-1,2) -- (1.3,1.9) -- (3,2);
 \draw[red] (-1,0) -- (4,0) ;
 \draw[red] (0,-1) -- (0,4);
 \draw[red] (1.3,1.9) node[circle, draw, fill=black!50,inner sep=0pt, minimum width=4pt] {};
 \draw[black] (1.3,2.1) node[left] {$\vecx_0$};
 \draw[black] (2,3) node[right] {$C$};

  \path[thick,->] (5,3) edge [bend left] (6.5,3);
  \draw[black] (5.75,3.3) node[above] {$h$};

 \draw[thick,black] (8,1) -- (10,1) -- (11,4) -- (9,4) -- (8,1);
 \draw[blue,dashed]  (8,1) -- (10,2) -- (11,4);
 \draw[blue,dashed]  (10,1) -- (10,2) -- (9,4);
 \draw[red] (7,0) -- (12,0) ;
 \draw[red] (8,-1) -- (8,4);
 \draw[red] (10,2) node[circle, draw, fill=black!50,inner sep=0pt, minimum width=4pt] {};
 \draw[black] (10,2.1) node[left] {$\vecy_0$};
 \draw[black] (11,4) node[right] {$\alpha(C)$};

\end{tikzpicture}
\caption{The locally piecewise affine map $h$ determined by $(C,\alpha,\vecx_0,\vecy_0)$.}\label{lpam-pic}
\end{center}
\end{figure}

It is clear that $h$ is necessarily continuous and invertible. Indeed the following result is straightforward.

\begin{lem}\label{lpam-inv} Let $h$ be the locally piecewise affine map determined by  $(C, \alpha, \vecx_0,\vecy_0)$. Then $h^{-1}$ is the locally piecewise affine map determined by  $(h(C), \alpha^{-1}, \vecy_0,\vecx_0)$.
\end{lem}

In the last section we shall repeatedly use the following special case of Lemma~\ref{vertex-move} applied with $\alpha = \mathrm{id}$, the identity mapping on $\mR^2$.

\begin{lem}\label{triangle-move} Suppose that $T$ and $\hat{T}$ are two triangles in $\mR^2$ with vertices $\veca,\vecb,\vecc$ and $\hat{\veca},\vecb,\vecc$ respectively.  
Suppose that $Q$ is a convex quadrilateral in $\mR^2$ which contains $T$ and $\hat{T}$ and which has $\overline{\vecb \vecc}$ as one side. Then the locally piecewise affine map determined by $(Q,\mathrm{id},\veca,\hat{\veca})$ maps $T$ onto $\hat{T}$ and fixes all points outside of $Q$.
\end{lem}

\begin{figure}[!ht]
\begin{center}
\begin{tikzpicture}[scale=0.8]
%
% Left picture
%
 \draw[thick,black] (0.5,2) -- (2,1) -- (1,4) -- (0.5,2);
 \draw[blue,dashed]  (0.5,2) -- (2,0) -- (4,2) -- (1,4) -- (0.5,2);
 \draw[red] (-1,0) -- (4,0) ;
 \draw[red] (0,-1) -- (0,4);
 \draw[red] (2,1) node[circle, draw, fill=black!50,inner sep=0pt, minimum width=4pt] {};
 \draw[red] (3,2) node[circle, draw, fill=black!50,inner sep=0pt, minimum width=4pt] {};
 \draw[black] (2,0.9) node[below] {$\veca$};
 \draw[black] (3,2) node[right] {$\hat{\veca}$}; 
 \draw[black] (0.5,2) node[left] {$\vecb$};
 \draw[black] (1,4) node[left] {$\vecc$};
 \draw[black] (0.7,2.3) node[right] {$T$};
 \draw[black] (2.5,3.2) node[right] {$Q$};

  \path[thick,->] (5,3) edge [bend left] (6.5,3);
  \draw[black] (5.75,3.3) node[above] {$h$};

 \draw[red] (7,0) -- (12,0) ;
 \draw[red] (8,-1) -- (8,4);
 \draw[thick,black] (8.5,2) -- (11,2) -- (9,4) -- (8.5,2);
 \draw[blue,dashed]  (8.5,2) -- (10,0) -- (12,2) -- (9,4) -- (8.5,2);
 \draw[red] (10,1) node[circle, draw, fill=black!50,inner sep=0pt, minimum width=4pt] {};
 \draw[red] (11,2) node[circle, draw, fill=black!50,inner sep=0pt, minimum width=4pt] {};
 \draw[black] (10,0.9) node[below] {$\veca$};
 \draw[black] (11,2) node[right] {$\hat{\veca}$}; 
 \draw[black] (8.5,2) node[left] {$\vecb$};
 \draw[black] (9,4) node[left] {$\vecc$};
 \draw[black] (9,2.6) node[right] {$\hat{T}$};
 \draw[black] (10.5,3.2) node[right] {$Q$};

\end{tikzpicture}
\caption{A locally piecewise affine map moving $T$ to $\hat{T}$.}\label{triangle-move-pic}
\end{center}
\end{figure}
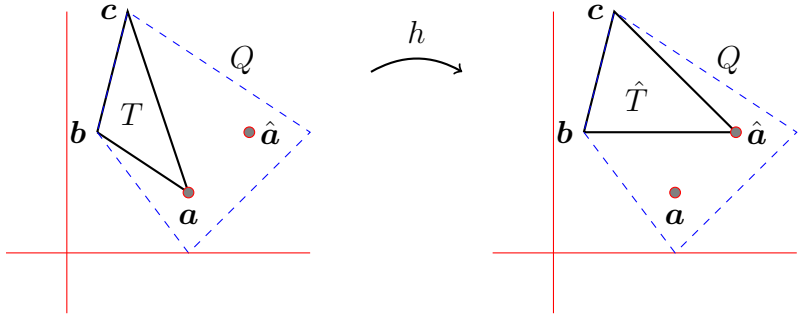

Our first aim is to show that for any locally piecewise affine map $h$, $\AC(\sigma) \simeq \AC(h(\sigma))$.

\begin{lem}\label{lpam-vf}
Suppose that $h$ is a locally piecewise affine map determined by $(\alpha,C,\vecx_0,\vecy_0)$ where $C$ is a convex $n$-gon.
Let $S = [\vecw_0,\vecw_1,\linebreak[1]\dots,\vecw_m]$ be a list of elements in $\mR^2$ and let $\hat S = [h(\vecw_0),h(\vecw_1),\dots,h(\vecw_m)]$. Then
  \[ \frac{1}{c_n} \vf(S) \le \vf({\hat S}) \le c_n \vf(S) \]
for some positive constant $c_n$ which is independent of $S$.
\end{lem}

\begin{proof}
By the previous lemma it suffices to prove either one of the inequalities. For notational simplicity, we shall write $s_i = \ls[\vecw_i,\vecw_{i+1}]$, for $0 \le i \le m-1$ and write $\vecv_i = h(\vecw_i)$ for $0 \le i \le m$. 

Let $T_1,\dots,T_n$ denote the subsets of the plane defined at the start of this section, and let $T_0 = \mR^2 \setminus \interior{C}$. 

Suppose then that $\vf(S) = k$ and that $\ell$ is a line such that $\vf(S,\ell) = \vf(S)$.
%There are $\binom{n+1}{2}$ pairs of the regions $T_0,\dots,T_n$.

%Each of the $k$ crossing edges of $S$ on $\ell$ has endpoints which sit in the union of at least one of these pairs. There must therefore be at least one pair of regions $T_{r_1}$ and $T_{r_2}$ which contain both endpoints for at least $k_1 = \lceil \frac{2 k}{n(n+1)}\rceil$ of these crossing edges.  Let $K = \{i \st x_i,x_{i+1} \in T_{r_1} \cup T_{r_2}\}$.

%Suppose first that neither $r_1$ nor $r_2$ is $0$.
At least one of the regions $T_0,\dots,T_n$ has at least $k_1 = \lceil \frac{k}{n+1} \rceil$ crossing segments of $S$ on $\ell$ with at least one of their endpoints in that region. Suppose that this region is $T_r$. Let $K$ denote the set of indices $i$ such that $s_i \in X(S,\ell)$ and $s_i$ has at least one endpoint in $T_r$.
Our aim is to show that each of these crossing segments corresponds to a crossing segment of $\hat S$ on one of a finite number of lines. This will require a careful consideration of cases.

Let
  \[ \delta = \min\{d(\vecv_i,h(T_r)) \st \text{$0 \le i \le n $ and $\vecv_i \not\in T_r$}\}. \]
We take the minimum of the empty set to be zero.

Suppose first that $1 \le r \le n$. Then $h(T_r)$ is a triangle. Choose a triangle $\hat T$ with sides parallel to those of $h(T_r)$, which contains $h(T_r)$ in its interior, and such that if $\vecv \in {\hat T}$ then $d(\vecv,h(T_r)) < \frac{\delta}{2}$.
Let $\ell_1$,$\ell_2$ and $\ell_3$ denote the three lines forming the sides of $\hat T$ (see Figure~\ref{l1-l2-l3}). From this construction,
% \footnote{If one just works with the boundary of $h(T_r)$ rather than this inflated version, this statement is not true. For example one can have $\vecw_0$ and $\vecw_1$ on one of the boundary lines of $T_r$ and $\vecw_2$ outside $T_r$. It is easy to draw examples where $\ls[\vecw_0,\vecw_1]$ and $\ls[\vecw_1,\vecw_2]$ are in $X({S},\ell)$, but where $\ls[\vecv_1,\vecv_2]$ is not a crossing segment for $\hat S$ on either $\ell_0$ or any of the lines bounding $h(T_r)$.}
 every segment ${\hat s}_i$ with $i \in K$ either lies entirely inside $h(T_r)$, or else it is a crossing segment for $\hat S$ on at least one of the lines $\ell_1$, $\ell_2$ or $\ell_3$.

Let $\ell_0$ denote the line which is the image of $\ell$ under $\alpha_r$ (considered as extended to the whole plane).

\begin{figure}[!ht]
\begin{center}
\begin{tikzpicture}[scale=0.8]
%
% Left picture
%
 \draw[thick,black] (0,0) -- (4,0) -- (3,5) -- (0,0);
 \draw[red,dashed] (-1.5,-0.3) -- (5.5,-0.3);
 \draw[red,dashed] (4.5,-1) -- (3.1,6);
 \draw[red,dashed] (-1,-1) -- (3.2,6);
 \draw[blue,thick] (-1,1) -- (5.5,4.5);

 \draw (0.0,0.7) node[left] {$\ell_1 $};
 \draw (2,-0.5) node[below] {$\ell_3$};
 \draw (4.4,2) node[below] {$\ell_2$};
 \draw (4.4,3.5) node[right] {$\ell_0$};
 \draw (1.5,4.7) node[right] {$\hat T$};
 \draw (2.3,1.3) node[below] {$h(T_r)$};
\end{tikzpicture}
\caption{The construction of $\ell_0$, $\ell_1$ ,$\ell_2$ and $\ell_3$.}\label{l1-l2-l3}
\end{center}
\end{figure}
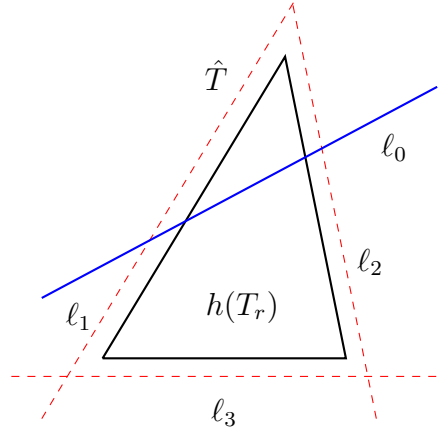

Suppose then that $i \in K$, and that both $\vecw_i$ and $\vecw_{i+1}$ lie in $T_r$. Referring to Definition~\ref{crossing-defn} there are four possibilities:
\begin{enumerate}[(i)]
 \item $\vecw_i$ and $\vecw_{i+1}$ lie on strictly opposite sides of $\ell$. In this case $\vecv_i$ and $\vecv_{i+1}$ lie on strictly opposite sides of $\ell_0$ and so ${\hat s}_i \in X({\hat S},\ell_0)$.
 \item $i=0$ and $\vecw_i \in \ell$. Clearly then $\vecv_i \in \ell_0$ and so ${\hat s}_i \in X({\hat S},\ell_0)$.
 \item $i > 0$, $\vecw_i \in \ell$ and $\vecw_{i-1} \not\in \ell$. In this case $s_{i-1} \not \in X(S,\ell)$ and either:
    \begin{enumerate}[(a)]
     \item $\vecw_{i-1} \in T_r$. In this case $\vecv_i \in \ell_0$ and $\vecv_{i-1} \not\in \ell_0$ and hence ${\hat s}_i \in X({\hat S},\ell_0)$.
     \item $\vecw_{i-1} \not \in T_r$. In this case ${\hat s}_i$ need not be in $X({\hat S},\ell_0)$ since $\vecv_{i-1}$ might lie on $\ell_0$. However, ${\hat s}_{i-1}$ must be a crossing segment for $\hat S$ on one of the boundary lines $\ell_1$, $\ell_2$ or $\ell_3$.
    \end{enumerate}
  \item $i=m-1$, $\vecw_i \not\in \ell$ and $\vecw_{i+1} \in \ell$. Again ${\hat s}_i \in X({\hat S},\ell_0)$.
\end{enumerate}
Suppose next that $i \in K$ and that one of $\vecw_i$ and $\vecw_{i+1}$ does not lie in $T_r$. As noted above, in this case ${\hat s}_i \in X({\hat S},\ell_j)$ for some $j=1,2,3$.

At this stage we have shown that if $1 \le r \le n$, then there are at least $k_1$ segments of $\hat S$ which are crossing segments for at least one of the lines $\ell_j$ with $0 \le j \le 4$. Thus, for at least one of these values of $j$, $\vf({\hat S},\ell_j) \ge \lceil \frac{k_1}{4} \rceil$.

The remaining case is where $r = 0$ and so $T_r$ is not a triangle.  The proof in this case is almost identical except that now one must work with 
\begin{enumerate}
 \item lines $\ell_1,\dots,\ell_n$ chosen close to the boundary of $h(T_0)$ so that all the endpoints $\vecv_i$ which are not in $T_0$ lie inside the smaller $n$-gon determined by these lines, and
 \item the line $\ell_0 = \alpha(\ell)$ (see Figure~\ref{T_0 case}).
\end{enumerate}
Every segment ${\hat s}_i = \ls[\vecv_i,\vecv_{i+1}]$ with only one endpoint in $T_0$ lies in $X({\hat S}, \ell_j)$ for at least one $j$ with $1 \le j \le n$. 

Following the proof above one can then show that there are (at least) $k_1$ segments of $\hat S$ which are crossing segments for at least one of the lines $\ell_j$ where $0 \le j \le n$, and hence $\vf({\hat S},\ell_j)  \ge  \lceil \frac{k_1}{n+1} \rceil$ for at least one value of $j$ in this range.

\begin{figure}[!ht]
\begin{center}
\begin{tikzpicture}[scale=0.8]
%
% Left picture
%
 \draw[thick,black] (0,0) -- (4,0) -- (5,4) -- (2,5) -- (-1,4) -- (0,0);
 \draw[thick,gray] (0,0) -- (3,3) -- (-1,4);
 \draw[thick,gray] (4,0) -- (3,3) -- (2,5);
 \draw[thick,gray] (3,3) -- (5,4);
 \draw[thick,blue]  (-1.5,1) -- (6,3);
 \draw[red,dashed] (-1,0.3) -- (6,0.3);
 \draw[red,dashed] (3.5,-1) -- (5,5);
 \draw[red,dashed] (0.5,-1) -- (-1,5);
 \draw[red,dashed] (-1.8,3.5) -- (4.2,5.5);
 \draw[red,dashed] (5.8,3.5) -- (-0.2,5.5);

 \draw (5.5,2.5) node[right] {$\ell_0 = \alpha(\ell)$};
 \draw (-3,2) node[left] {$h(T_0)$};
 \draw (5.7,0.3) node[above] {$\ell_1$};
 \draw (3.5,-1) node[right] {$\ell_2$};
 \draw (0.5,-1) node[left] {$\ell_5$};
 \draw (0,5.4) node[above] {$\ell_3$};
 \draw (4,5.4) node[above] {$\ell_4$};

\end{tikzpicture}
\caption{Choosing $\ell_0,\dots,\ell_n$ when $r = 0$ (and $n=5$).}\label{T_0 case}
\end{center}
\end{figure}
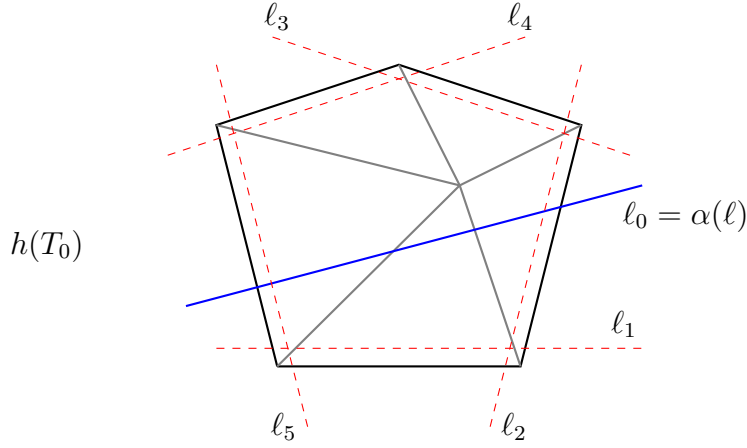

In either case then
  \[ \vf({\hat S}) \ge \left \lceil \frac{k_1}{n+1} \right \rceil
    = \left \lceil \frac{\lceil\, \vf(S)/(n+1)\, \rceil}{n+1} \right \rceil
    \ge \frac{\vf(S)}{(n+1)^2}. \]
\end{proof}

The above proof of course shows that $c_n \le (n+1)^2$. Heuristically we expect that $c_n = n+1$ but we are unable to prove this.

\begin{thm}\label{lpam-isom}
Suppose that $\sigma$ is a nonempty compact subset of the plane, and that $h$ is a locally piecewise affine map. Then $\BV(\sigma) \simeq \BV(h(\sigma))$ and $\AC(\sigma) \simeq \AC(h(\sigma))$
\end{thm}

\begin{proof}
For $f \in \BV(\sigma)$, let ${\hat f}: h(\sigma) \to \mC$ be defined by ${\hat f}(h(\vecx)) = f(\vecx)$. Suppose that $f \in \BV(\sigma)$ and that ${\hat S} = [\vecv_0,\dots,\vecv_m]$ is a list of points in $h(\sigma)$. Let $S = [\vecw_0,\dots,\vecw_m] \subset \sigma$ denote the list of preimages of the points in $\hat S$. Then, using the notation of Lemma~\ref{lpam-vf},
  \[ \frac{\cvar({\hat f},{\hat S})}{\vf({\hat S})}
     =\frac{\cvar({f},{S})}{\vf({\hat S})}
     \le C_n \frac{\cvar({f},{S})}{\vf({S})} \le C_n \var(f,\sigma).
  \]
Thus, $\hat f$ is of bounded variation with $\ssnorm{\hat f}_{\BV(h(\sigma))} \le (1+C_n) \normbv{f}$. It follows, using Lemma~\ref{lpam-inv} that the map $j: f \mapsto {\hat f}$ is a bounded isomorphism from $\BV(\sigma)$ onto $\BV(h(\sigma))$.

It is clear that $j$ maps $\CTPP(\sigma)$ onto $\CTPP(h(\sigma))$ and hence that $j$ provides an isomorphism from $\AC(\sigma)$ onto $\AC(h(\sigma))$.
\end{proof}

\section{Polygons and ears}

The main result from this section is that given any two simple polygons $P_1$ and $P_2$ we have that $\AC(P_1) \simeq \AC(P_2)$. The proof requires a nice fact from computational geometry called the `Two Ears Theorem' which was proven by Meisters \cite{Me}.

Given $\veca,\vecb \in \mR^2$ we shall let $\lso[\veca,\vecb]$ denote the `open' line segment between $\veca$ and $\vecb$, that is
  \[ \lso[\veca,\vecb] = \{ \lambda \veca + (1-\lambda) \vecb \st 0 < \lambda < 1\}. \]

% Let $P$ be a (simple) polygon with vertices $\vecv_1,\dots,\vecv_n$. A vertex $\vecv_j$ is called an ear of $P$ if the diagonal $\lso[\vecv_{j-1} , \vecv_{j+1}]$ lies entirely in the interior of $P$.

Let $\vecv$ be a vertex of a polygon $P$ and suppose that $\veca,\vecb$ are the neighbouring vertices to $\vecv$. We say that $\vecv$ is an ear of $P$ if $\lso[\veca , \vecb]$ lies entirely in the interior of $P$.

\begin{thm}[Two Ears Theorem] Every simple polygon with more than 3 vertices has at least 2 ears.
\end{thm}

A simple consequence of the Two Ears Theorem is that it is possible to triangulate any polygon. That is, given any polygon $P$, one may construct a finite family of `disjoint' triangles $\{T_n\}$ whose vertices are all vertices of $P$ and whose union equals $P$.

\begin{figure}[!ht]
\begin{center}
\begin{tikzpicture}[scale=1]
%
% Left picture
%
 \fill[blue!15] (1,3.2) -- (2,4.5) -- (4,3);
 \draw[thick,black] (0,0) -- (4,3) -- (2,4.5) -- (1,3.2) -- (1,1.7) -- (-1,4.5) -- (1.7,6) -- (3,4.5) -- (2,7) -- (-2,5) -- (0,0);
 \draw[black,dashed] (1,3.2) -- (4,3);
 \draw[thick,red,dashed] (1,3.2) -- (1.6,5.4) -- (4,3) -- (2.1,2.3) -- (1,3.2);

 \draw[red] (1.6,5.4) node[circle, draw, fill=black!50,inner sep=0pt, minimum width=4pt] {};
 \draw[black] (1.6,5.4) node[left] {$\vecu$};
 \draw[red] (2,4.5) node[circle, draw, fill=black!50,inner sep=0pt, minimum width=4pt] {};
 \draw[black] (2.05,4.5) node[below] {$\vecv$};
 \draw[red] (1,3.2) node[circle, draw, fill=black!50,inner sep=0pt, minimum width=4pt] {};
 \draw[black] (1,3.2) node[left] {$\veca$};
 \draw[red] (4,3) node[circle, draw, fill=black!50,inner sep=0pt, minimum width=4pt] {};
 \draw[black] (4,3) node[right] {$\vecb$};
 \draw[red] (2.1,2.3) node[circle, draw, fill=black!50,inner sep=0pt, minimum width=4pt] {};
 \draw[black] (2.1,2.3) node[below] {$\vecw$};

 \draw[black] (0,2.4) node[below] {$P$};
 \draw[black] (2.4,3.1) node[above] {$T$};
 \draw[black] (1.3,4.2) node[left] {$Q$};

%  \draw[black] (6.4,-0.6) node[below] {$T_1$};
%  \draw[red,dashed] (2.8,6) -- (4.4,-2);
%  \draw[red,dashed] (0,0) -- (4.2,-1) -- (8,0);
%  \draw[red] (4,0) node[circle, draw, fill=black!50,inner sep=0pt, minimum width=4pt] {};
%  \draw[red] (4.2,-1) node[circle, draw, fill=black!50,inner sep=0pt, minimum width=4pt] {};
\end{tikzpicture}
\caption{Lemma~\ref{construct-Q}.}
\end{center}
\end{figure}

\begin{lem}\label{construct-Q}
Suppose that $\vecv$ is an ear of a polygon $P$. Let $T = \triangle \veca \vecv \vecb$ denote the triangle formed by the two sides of $P$ which meet at $\vecv$ and the corresponding diagonal $\ls[\veca, \vecb]$. Then there exists a convex quadrilateral $Q$ with vertices $\veca$, $\vecu$, $\vecb$ and $\vecw$ such that
\begin{enumerate}
 \item $T \setminus \{\veca,\vecb\} \subseteq \interior{Q}$,
 \item $\lso[\veca,\vecu]$ and $\lso[\vecb,\vecu]$ lie in the complement of $P$, and
 \item $\lso[\veca,\vecw]$ and $\lso[\vecb,\vecw]$ lie in the interior of $P$.
\end{enumerate}
\end{lem}

\begin{proof} We begin by triangulating $P$ using the standard algorithm of removing one ear at a time. We may clearly start by dealing with the ear at $\vecv$ and so $T$ is one of the triangles in our triangulation.  The diagonal $\ls[\veca,\vecb]$ must form an edge of two of the triangles, namely $T$, and another which we shall denote by $T_1$. One can choose $\vecw$ to be any interior point of $T_1$ and this will clearly have property~3.  Let $\vecm$ denote the midpoint of $\ls[\veca,\vecb]$ and let $\ell$ denote the median of $T$ that passes through $\vecv$ and $\vecm$. As we shall see below, if we demand that $\vecw$ also lies on $\ell$ then this will ensure that the quadrilateral $Q$ is convex. (See Figure~\ref{get-w}.)

\begin{figure}[!ht]
\begin{center}
\begin{tikzpicture}[scale=0.8]
%
% Left picture
%
 \draw[thick,black] (-2,-2) -- (0.5,-1) -- (0,0) -- (3,5) -- (8,0) -- (6,-2) -- (8,-2);
 \draw[blue]  (0,0) -- (7.98,0) -- (5.98,-2) -- (0,0);
 \draw[black] (5,2) node[below] {$T$};
 \draw[black] (6.4,-0.6) node[below] {$T_1$};
 \draw[red,dashed] (2.8,6) -- (4.4,-2);
 \draw[red,dashed] (0,0) -- (4.2,-1) -- (8,0);
 \draw[red] (4,0) node[circle, draw, fill=black!50,inner sep=0pt, minimum width=4pt] {};
 \draw[red] (4.2,-1) node[circle, draw, fill=black!50,inner sep=0pt, minimum width=4pt] {};
 \draw[black] (0,0) node[left] {$\veca$};
 \draw[black] (3,5) node[right] {$\vecv$};
 \draw[black] (8,0) node[right] {$\vecb$};
 \draw[black] (3.2,3) node[left] {$\ell$};
 \draw[black] (3.9,0.2) node[left] {$\vecm$};
 \draw[black] (4.3,-1.2) node[right] {$\vecw$};
\end{tikzpicture}
\caption{Construction of $\vecw$.}\label{get-w}
\end{center}
\end{figure}
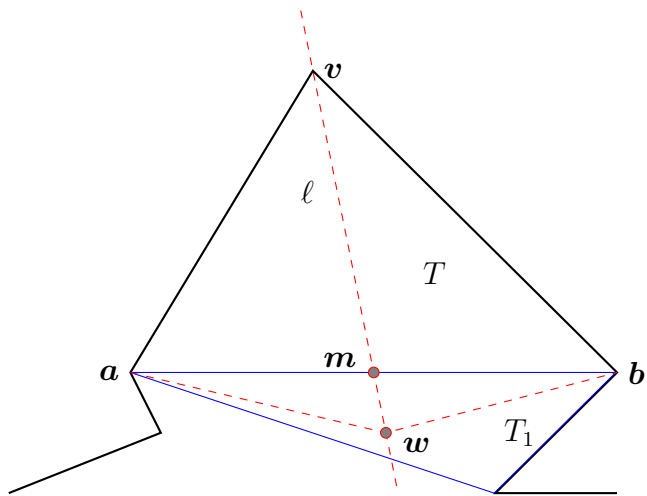

Finding a suitable point $\vecu$ is slightly more delicate. For small $t > 0$, $\vecu(t) = (1+t) \vecv - t \vecw$ lies in the complement of $P$. If $\lso[\veca,\vecu(t)]$ or $\lso[\vecb,\vecu(t)]$ does not lie in the complement of $P$ then it must be the case that some vertices of $P$ lie in the interior of the quadrilateral $\veca \vecv \vecb \vecu(t)$ or on one of the boundary lines $\ls[\veca,\vecu(t)]$ or $\ls[\vecb,\vecu(t)]$. Since there can be only finitely many such vertices, by choosing $t_0 > 0$  sufficiently small we can ensure that $\vecu = \vecu(t_0)$ satisfies condition (2).

Property~(1) is clear so it remains to check convexity. By the construction, the two diagonals of $Q$ will meet at $\vecm$ and clearly $\vecm \in \lso[\veca,\vecb]$ and  $\vecm \in \lso[\vecu,\vecw]$. But by Theorem~6.7.9 of \cite{V} a quadrilateral is convex if and only if the diagonals meet at a point in the interior of these diagonals, and hence $Q$ is convex.
\end{proof}

\begin{thm}\label{polygons}
Suppose that $P_1$ and $P_2$ are simple polygons. Then $\AC(P_1) \simeq \AC(P_2)$.
\end{thm}

\begin{proof}
We shall use induction to prove that
% It is clear that the algebras for any two triangles are isometrically isomorphic. It suffices therefore, to prove that
the statement
  \[ S(n): \ \hbox{if $P$ is any simple $n$-gon and $T$ is a triangle, then $\AC(P) \simeq \AC(T)$} \]
holds for all $n \ge 3$.

The statement is true for $n = 3$ since one can find an affine map between any two triangles. Suppose then that $n > 3$ and that the $S(m)$ is true for all $m$ with $3 \le m < n$. Let $P$ be an $n$-gon with vertices $\vecv_1,\dots,\vecv_n$. By the Two Ears Theorem, there exists an ear $\vecv_j$. Let $T_{\vecv_j}$ be the triangle with vertices at $\vecv_{j-1}$, $\vecv_j$ and $\vecv_{j+1}$.

Using Lemma~\ref{construct-Q}
fix a convex quadrilateral $Q$ with vertices at $\vecv_{j-1}$ and $\vecv_{j+1}$ and two additional points $\vecu, \vecw$ chosen so that $T_{\vecv_j}\setminus \{\vecv_{j-1},\vecv_{j+1}\}$ lies in the interior of $Q$ and so that $\lso[\vecv_{j-1},\vecw]$ and $\lso[\vecv_{j+1},\vecw]$ lie in the interior of $P$. As $Q$ is convex, the point $\vecv_j$ and the midpoint $\vecm$ of $\ls[\vecv_{j-1} ,\vecv_{j+1}]$ are both interior points of $Q$. 

% Using Lemma~\ref{construct-Q}
% fix a convex quadrilateral $Q$ with vertices at $\vecv_{j-1}$ and $\vecv_{j+1}$ and two additional points $\vecu ,\vecw$ chosen so that, apart from the %interior of the 
% triangle $\vecv_{j-1}\, \vecv_j\, \vecv_{j+1}$ lies in the interior of $Q$
% and so that $\lso[\vecv_{j-1},\vecw]$ and $\lso[\vecv_{j+1},\vecw]$ lie in the inteior of $P$. As $Q$ is convex, the midpoint $\vecm$ of $\ls[\vecv_{j-1} ,\vecv_{j+1}]$ is an interior point of $Q$. 

%
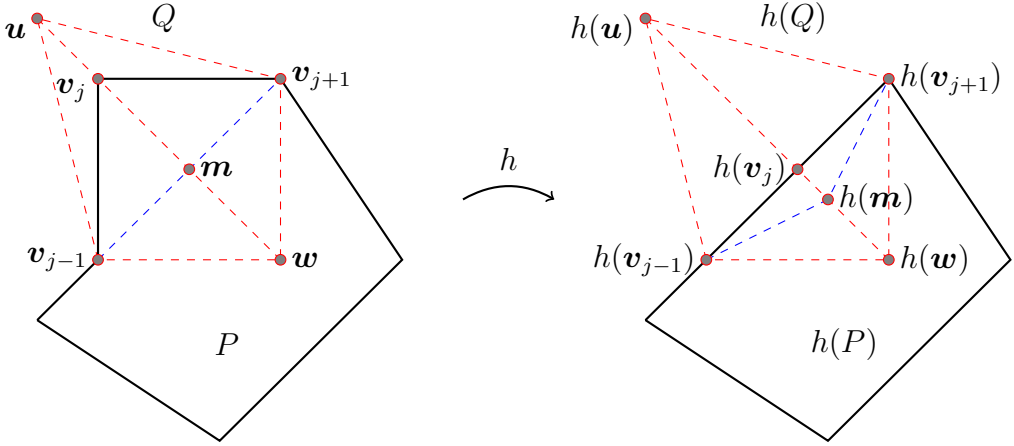
\begin{figure}[!ht]
\begin{center}
\begin{tikzpicture}[scale=0.8]
%
% Left picture
%
 \draw[thick,black] (-1,-1) -- (0,0) -- (0,3) -- (3,3) -- (5,0) -- (2,-3) -- (-1,-1);
 \draw[blue,dashed]  (0,0) -- (3,3);
 \draw[red,dashed] (0,0) -- (-1,4) -- (3,3) -- (3,0) -- (0,0);
 \draw[red,dashed] (-1,4) -- (3,0);
 \draw[red] (0,3) node[circle, draw, fill=black!50,inner sep=0pt, minimum width=4pt] {};
 \draw[black] (0,2.8) node[left] {$\vecv_j$};
 \draw[red] (0,0) node[circle, draw, fill=black!50,inner sep=0pt, minimum width=4pt] {};
 \draw[black] (0,0) node[left] {$\vecv_{j-1}$};
 \draw[red] (3,3) node[circle, draw, fill=black!50,inner sep=0pt, minimum width=4pt] {};
 \draw[black] (3,3) node[right] {$\vecv_{j+1}$};
 \draw[red] (-1,4) node[circle, draw, fill=black!50,inner sep=0pt, minimum width=4pt] {};
 \draw[black] (-1,3.8) node[left] {$\vecu$};
 \draw[red] (1.5,1.5) node[circle, draw, fill=black!50,inner sep=0pt, minimum width=4pt] {};
 \draw[black] (1.5,1.5) node[right] {$\vecm$};
 \draw[red] (3,0) node[circle, draw, fill=black!50,inner sep=0pt, minimum width=4pt] {};
 \draw[black] (3,0) node[right] {$\vecw$};
 \draw[black] (2.5,-1.4) node[left] {$P$};
 \draw[black] (0.7,4) node[right] {$Q$};

  \path[thick,->] (6,1) edge [bend left] (7.5,1);
  \draw[black] (6.75,1.3) node[above] {$h$};

 \draw[thick,black] (9,-1) -- (10,0) -- (13,3) -- (15,0) -- (12,-3) -- (9,-1);
 \draw[blue,dashed]  (10,0) -- (12,1) -- (13,3);
 \draw[red,dashed] (10,0) -- (9,4) -- (13,3) -- (13,0) -- (10,0);
 \draw[red,dashed] (9,4) -- (13,0);
 \draw[red] (11.5,1.5) node[circle, draw, fill=black!50,inner sep=0pt, minimum width=4pt] {};
 \draw[black] (11.5,1.5) node[left] {$h(\vecv_j)$};
 \draw[red] (10,0) node[circle, draw, fill=black!50,inner sep=0pt, minimum width=4pt] {};
 \draw[black] (10,0) node[left] {$h(\vecv_{j-1})$};
 \draw[red] (13,3) node[circle, draw, fill=black!50,inner sep=0pt, minimum width=4pt] {};
 \draw[black] (13,3) node[right] {$h(\vecv_{j+1})$};
 \draw[red] (9,4) node[circle, draw, fill=black!50,inner sep=0pt, minimum width=4pt] {};
 \draw[black] (9,3.8) node[left] {$h(\vecu)$};
 \draw[red] (12,1) node[circle, draw, fill=black!50,inner sep=0pt, minimum width=4pt] {};
 \draw[black] (12,1) node[right] {$h(\vecm)$};
 \draw[red] (13,0) node[circle, draw, fill=black!50,inner sep=0pt, minimum width=4pt] {};
 \draw[black] (13,0) node[right] {$h(\vecw)$};
 \draw[black] (13,-1.4) node[left] {$h(P)$};
 \draw[black] (10.7,4) node[right] {$h(Q)$};
\end{tikzpicture}
\caption{The action of $h$.}\label{cut-ear}
\end{center}
\end{figure}

Let $\alpha$ denote the identity map on $\mR^2$, and let $h$ denote the unique locally piecewise affine map determined by $(Q,\alpha,\vecv_j,\vecm)$.
This map sends $\ls[\vecv_{j-1} ,\vecv_j]$ to $\ls[\vecv_{j-1} ,\vecm]$ and $\ls[\vecv_j ,\vecv_{j+1}]$ to $\ls[\vecm ,\vecv_{j+1}]$ (see Figure~\ref{cut-ear}). All the other edges of $P$ lie in the complement of $Q$ and are therefore fixed. It follows therefore that the image of $P$ under $h$ is an $m$-gon for some $m < n$. By Theorem~\ref{lpam-isom}, $\AC(P) \simeq \AC(h(P))$. But by the induction hypothesis, $\AC(h(P)) \simeq \AC(T)$ for any triangle $T$ and so the proof is complete.
\end{proof}

\section{Polygonal regions with holes}\label{poly-with-holes}

The results of the last section have a natural extension to a wider class of regions. Let $P$ be a simple polygon in the plane. A set $W$ is a \textbf{window} in $P$ if it is the interior of a polygon $P'$ where $P'$ lies in the interior of $P$.
We shall say that a compact set $\sigma$ is a \textbf{polygonal region of genus $n$} if there exists a simple polygon $P$ with $n$ nonoverlapping  windows $W_1,\dots,W_n$ such that
  \[ \sigma = P \setminus (W_1 \cup \dots \cup W_n)  \]
and write $G(\sigma) = n$ for the genus of $\sigma$.

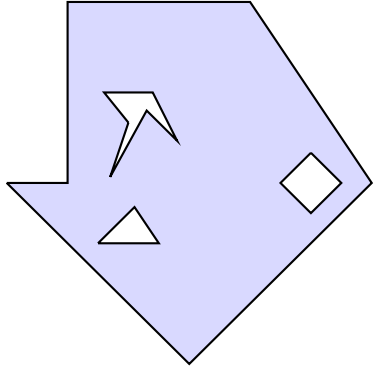
\begin{figure}[!ht]
\begin{center}
\begin{tikzpicture}[scale=0.8]
%
% Left picture
%
 \fill[blue!15] (-1,0) -- (0,0) -- (0,3) -- (3,3) -- (5,0) -- (2,-3) -- (-1,0);
 \draw[thick,black] (-1,0) -- (0,0) -- (0,3) -- (3,3) -- (5,0) -- (2,-3) -- (-1,0);
 \fill[white] (1,1) -- (0.6,1.5) -- (1.4,1.5) -- (1.8,0.7) -- (1.3,1.2) -- (0.7,0.1) -- (1,1);
 \draw[thick,black]  (1,1) -- (0.6,1.5) -- (1.4,1.5) -- (1.8,0.7) -- (1.3,1.2) -- (0.7,0.1) -- (1,1);
 \fill[white] (0.5,-1) -- (1.5,-1) -- (1.1,-0.4) -- (0.5,-1);
 \draw[thick,black] (0.5,-1) -- (1.5,-1) -- (1.1,-0.4) -- (0.5,-1);
 \fill[white] (4,0.5) -- (4.5,0) -- (4,-0.5) -- (3.5,0) -- (4,0.5);
 \draw[thick,black] (4,0.5) -- (4.5,0) -- (4,-0.5) -- (3.5,0) -- (4,0.5);
\end{tikzpicture}
\caption{A polygonal region of genus $3$.}
\end{center}
\end{figure}

If $\sigma_1$ and $\sigma$ two are polygonal regions of differing genus, then these sets are not homeomorphic and hence $\AC(\sigma_1) \not\simeq \AC(\sigma_2)$.
In this section we shall show that show within this class of sets, the isomorphism class of the the corresponding function algebras is completely determined by their genus. This is achieved by showing that there is always a finite sequence of locally piecewise affine maps whose composition sends $\sigma_1$ to $\sigma$, and then applying Theorem~\ref{lpam-isom}. One of the main tools in doing this is to show that via such maps, one may `move' triangular windows anywhere within any rectangle that contains no other other windows.

\begin{lem}\label{TinR}
Suppose that $R$ is a rectangle and that $T = \triangle\veca\vecb\vecc$ and $T' = \triangle \veca' \vecb'\vecc'$ are two triangles in the interior of $R$. Then there is a continuous bijection $h: \mR^2 \to \mR^2$ such that
\begin{enumerate}
 \item $h$ can be written as a composition of finitely many locally piecewise affine maps,
 \item $h(\vecx) = \vecx$ for all $\vecx \not\in R$,
 \item $h(R) = R$, and
 \item $h(T) = T'$.
\end{enumerate}
\end{lem}

\begin{proof} Choose $\epsilon > 0$ such that no point of $T$ or $T'$ lies within distance $2 \epsilon$ of the boundary of $R$. We shall call the four interior points of $R$ which lie at distance $\epsilon$ along the diagonals from the vertices of $R$, the $\epsilon$-corner points of $R$. Fix any three of these $\epsilon$-corner points and let $T_0$ denote the triangle with these points as vertices. We shall show that there is a function $h$ satisfying (1), (2) and (3) and such that $h(T) = T_0$. The same proof of course would construct a corresponding map sending $T'$ to $T_0$. Since the inverse of a locally piecewise affine map is also locally piecewise affine, this produces a finite sequence of locally piecewise affine maps which has properties (1) -- (4).

The line through $\vecb$ and $\vecc$ splits $R$ into two convex polygons.  Let $P$ denote the polygon which contains $\veca$. At least one of the vertices of $P$, say $\vecv_1$,  is a vertex of $R$ not lying on the line through $\vecb$ and $\vecc$. 

Let $\veca_1$ be the $\epsilon$-corner point of $R$ near $\vecv_1$. 
Using the triangulations of $P$ generated by $\veca$ and by $\veca_1$,  Lemma~\ref{vertex-move} produces a locally piecewise affine map $h_1$ which is the identity outside of $P$, and which maps $\veca$ to $\veca_1$. Indeed, as $T$ lies entirely in a region on which $h_1$ is affine, $h_1$ maps $T$ to the triangle $\triangle \veca_1 \vecb \vecc$ (see Figure~\ref{move-vert-1}).

\begin{figure}[!ht]
\begin{center}
\begin{tikzpicture}[scale=0.8]
%
% Left picture
%
 \fill[blue!15]  (0,0) -- (0,5) -- (4,5) --  (4,0) -- (0,0);
 \draw[thick,black] (0,0) -- (0,5) -- (4,5) --  (4,0) -- (0,0); 
 \fill[white] (2,2) -- (3,3) -- (1.3,3.2) -- (2,2);
 \draw[thick,black]  (2,2) -- (3,3) -- (1.3,3.2) -- (2,2);
 \draw[thick,blue]  (-0.5,-0.5) -- (4.5,4.5);
 \draw[red,dashed] (0,0) -- (1.3,3.2);
 \draw[red,dashed] (0,5) -- (1.3,3.2);
 \draw[red,dashed] (4,5) -- (1.3,3.2);
 \draw[red,dashed] (4,4) -- (1.3,3.2);
  \draw[red] (2,2) node[circle, draw, fill=black!50,inner sep=0pt, minimum width=4pt] {};
   \draw[red] (3,3) node[circle, draw, fill=black!50,inner sep=0pt, minimum width=4pt] {};
  \draw[red] (1.3,3.2) node[circle, draw, fill=black!50,inner sep=0pt, minimum width=4pt] {};
  \draw[red] (0,5) node[circle, draw, fill=black!50,inner sep=0pt, minimum width=4pt] {};
 \draw[black] (2,2) node[right] {$\vecb$};
 \draw[black] (3,3) node[right] {$\vecc$};
 \draw[black] (1,3.2) node[left] {$\veca$};
  \draw[black] (0,5) node[left] {$\vecv_1$};
 \draw[black] (2,4.3) node[left] {$P$};

   \path[thick,->] (5.5,4) edge [bend left] (7.0,4);
  \draw[black] (6.25,4.3) node[above] {$h_1$};
  
  \fill[blue!15]  (8,0) -- (8,5) -- (12,5) --  (12,0) -- (8,0);
 \draw[thick,black] (8,0) -- (8,5) -- (12,5) --  (12,0) -- (8,0);
 \fill[white] (10,2) -- (11,3) -- (8.3,4.65) -- (10,2);
 \draw[thick,black] (10,2) -- (11,3) -- (8.3,4.65) -- (10,2);
 \draw[thick,blue]  (7.5,-0.5) -- (12.5,4.5);
 \draw[red,dashed] (8,0) -- (8.3,4.65);
 \draw[red,dashed] (8,5) -- (8.3,4.65);
 \draw[red,dashed] (12,5) -- (8.3,4.65);
 \draw[red,dashed] (12,4) -- (8.3,4.65);
   \draw[red] (10,2) node[circle, draw, fill=black!50,inner sep=0pt, minimum width=4pt] {};
   \draw[red] (11,3) node[circle, draw, fill=black!50,inner sep=0pt, minimum width=4pt] {};
    \draw[red] (8.3,4.65) node[circle, draw, fill=black!50,inner sep=0pt, minimum width=4pt] {};
  \draw[red] (8,5) node[circle, draw, fill=black!50,inner sep=0pt, minimum width=4pt] {};
 \draw[black] (10,2) node[right] {$\vecb$};
 \draw[black] (11,3) node[right] {$\vecc$};
  \draw[black] (8,5) node[left] {$\vecv_1$};
 \draw[black] (8.3,4.45) node[below] {$\veca_1$};
\end{tikzpicture}
\caption{Moving the first vertex in Lemma~\ref{TinR}.}\label{move-vert-1}
\end{center}
\end{figure}
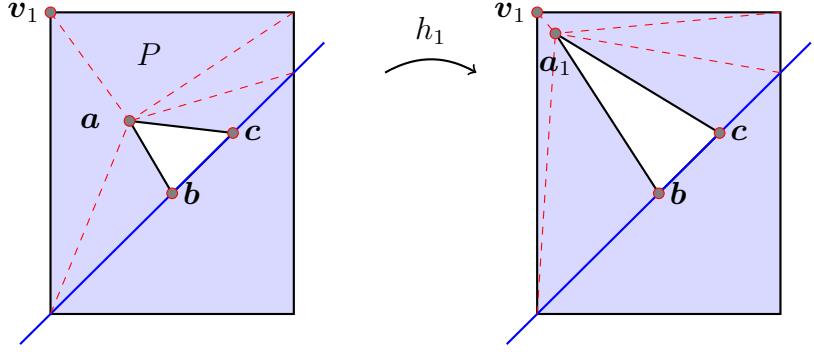

Consider now the quadrilateral $Q$ with vertices at $\veca_1$ and the three vertices of $R$ other than $\vecv_1$. Note that the position of $\veca_1$ ensures that $Q$ is convex, and hence that the line through $\veca_1$ and $\vecc$ splits $Q$ into two convex polygons. Let $P_1$ denote the polygon containing $\vecb$. One of the vertices of $Q$ adjacent to $\veca_1$ (which is therefore also a vertex of $R$) must lie in $P_1$. Denote this vertex by $\vecv_2$ and let $\vecb_1$ be the $\epsilon$-corner point of $R$ near $\vecv_2$. (Note that $\vecb_1$ must lie in $P_1$.)
 Applying Lemma~\ref{vertex-move} again we produce a locally piecewise affine map $h_2$ which is the identity outside of $Q$ and which maps $\triangle \veca_1 \vecb \vecc$ onto $\triangle \veca_1 \vecb_1 \vecc$.

Finally, consider the convex quadrilateral $Q_1$ with vertices $\veca_1,\vecb_1$ and the two remaining vertices of $R$. Let  $\vecc_1$ be the $\epsilon$-corner point of $R$ near one of these remaining vertices of $R$.  Noting that $\vecc$ and $\vecc_1$ are both in the interior of $Q_1$ we can find a locally piecewise affine map $h_3$ which is the identity outside of $Q_1$ and which maps $\triangle \veca_1 \vecb_1 \vecc$ onto $\triangle \veca_1 \vecb_1 \vecc_1$.

The vertices of this final triangle are all $\epsilon$-corner points. With one or two further applications of locally piecewise affine maps we can arrange that the image of $T$ under this composition of maps is $T_0$.
\end{proof}

\begin{thm}
Suppose that $\sigma_1$ and $\sigma_2$ are polygonal regions of genus $n_1$ and $n_2$. Then $\AC(\sigma_1) \simeq \AC(\sigma_2)$ if and only if $n_1 = n_2$.
\end{thm}

\begin{proof}
As noted above it only remains to show the `if' part of the theorem. Fix a genus $n$.
Let $\tau$ denote the polygonal region of genus $n$
  \[ \tau = T \setminus (T_1 \cup \dots \cup T_n)  \]
where $T$ is the triangle with vertices at $(0,-1)$, $(1,0)$ and $(0,1)$ and, for $k = 1,\dots,n$, the window $T_n$ is the triangle with vertices at $(\frac{3k-2}{3n},0)$, $(\frac{3k-1}{3n},0)$ and $(\frac{2k-1}{2n},\frac{1}{3n})$. We shall proceed by showing that if $\sigma$ is any polygonal regions of genus $n$, then $\AC(\sigma) \simeq \AC(\tau)$.

\begin{figure}[!hb]
\begin{center}
\begin{tikzpicture}[scale=0.6]
%
% Left picture
%
 \fill[blue!15] (-2,-5) -- (6,0) -- (-2,5) -- (-2,-5);
 \draw[very thick,black] (-2,-5) -- (6,0) -- (-2,5) -- (-2,-5);

 \fill[white] (-1,0) -- (0,0) -- (0,1.5) -- (1.5,1.5) -- (3.5,0) -- (1,-1) -- (-1,0);
 \draw[thick,black] (-1,0) -- (0,0) -- (0,1.5) -- (1.5,1.5) -- (3.5,0) -- (1,-1) -- (-1,0);
 \draw[red,dashed] (0,0) -- (-0.5,2) -- (1.5,1.5) -- (1,0.5) -- (0,0);
 \draw[red] (0,1.5) node[circle, draw, fill=black!50,inner sep=0pt, minimum width=4pt] {};
 \draw[black] (0,1) node[left] {$\vecv_j$};
 \draw[black] (2.5,0) node[left] {$V$};
 \draw[black] (0.7,1.6) node[above] {$Q$};

 \fill[white] (-1.5,3) -- (-0.5,3.5) -- (-1.5,2.5) -- (-1.5,3);
 \draw[thick,black] (-1.5,3) -- (-0.5,3.5) -- (-1.5,2.5) -- (-1.5,3);
 \fill[white] (-1,-2) -- (0,-2) -- (0,-3) -- (-1,-3) -- (-1,-2);
 \draw[thick,black] (-1,-2) -- (0,-2) -- (0,-3) -- (-1,-3) -- (-1,-2);

  \path[thick,->] (6,1) edge [bend left] (7.5,1);
  \draw[black] (6.75,1.3) node[above] {$h$};

 \fill[blue!15] (8,-5) -- (16,0) -- (8,5) -- (8,-5);
 \draw[very thick,black] (8,-5) -- (16,0) -- (8,5) -- (8,-5);
 \fill[white] (9,0) -- (10,0) -- (11.5,1.5) -- (13.5,0) -- (11,-1) -- (9,0);
 \draw[thick,black] (9,0) -- (10,0) -- (11.5,1.5) -- (13.5,0) -- (11,-1) -- (9,0);
 \draw[red,dashed] (10,0) -- (9.5,2) -- (11.5,1.5) -- (11,0.5) -- (10,0);
 \draw[red] (10.75,0.75) node[circle, draw, fill=black!50,inner sep=0pt, minimum width=4pt] {};
 \draw[black] (10.75,0.85) node[left] {$h(\vecv_j)$};
 \draw[black] (13,0) node[left] {$h(V)$};
 \draw[black] (10.7,1.6) node[above] {$h(Q)$};

 \fill[white] (8.5,3) -- (9.5,3.5) -- (8.5,2.5) -- (8.5,3);
 \draw[thick,black] (8.5,3) -- (9.5,3.5) -- (8.5,2.5) -- (8.5,3);
 \fill[white] (9,-2) -- (10,-2) -- (10,-3) -- (9,-3) -- (9,-2);
 \draw[thick,black] (9,-2) -- (10,-2) -- (10,-3) -- (9,-3) -- (9,-2);

\end{tikzpicture}
\caption{Reducing the number of edges in a window.}\label{reduce-sides}
\end{center}
\end{figure}
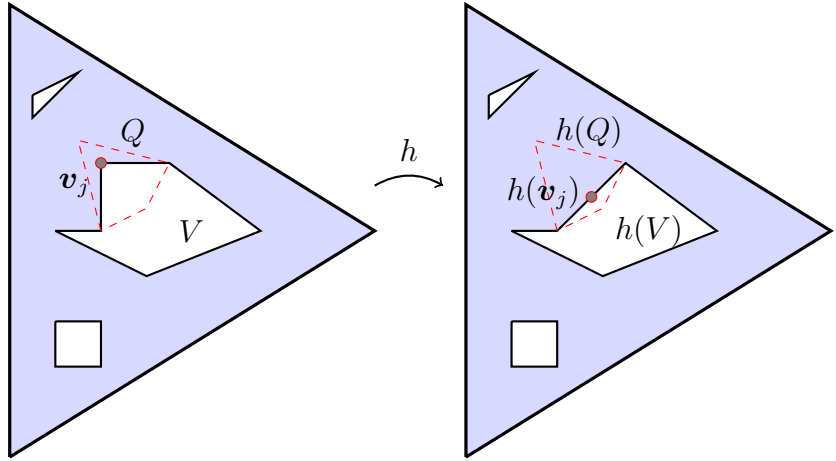

Suppose then that
  \[
   \sigma =P \setminus (W_1 \cup \dots \cup W_n).  \]
The image of $\sigma$ under any locally piecewise affine map is also a polygonal region of genus $n$ and as before, the isomorphism class of the corresponding $\AC$ function space is preserved under such maps.
By applying a finite sequence of locally piecewise affine maps as in the proof of Theorem~\ref{polygons} we may reduce the number of vertices in $P$ to 3. Note that the effect of these maps might be to increase the number of vertices in some of the windows. By applying a suitable affine map then, we see that $\AC(\sigma) \simeq \AC(\sigma')$ where
  \[ \sigma' =T \setminus (V_1 \cup \dots \cup V_n), \]
and where $V_1,\dots,V_n$ are windows in $T$.

The same algorithm can now be used to reduce the windows $V_1,\dots,V_n$ to triangles. Specifically, suppose that $V$ is a window in $T$ with at least $4$ vertices $\vecv_1,\dots,\vecv_k$. By the Two Ears Theorem we can choose an ear $\vecv_j$ in $V$. Since $\sigma'$ can be triangulated, the proof of Theorem~\ref{construct-Q} allows us to choose a convex quadrilateral $Q$ containing the triangular region $\vecv_{j-1}\, \vecv_j \vecv_{j+1}$ but not intersecting any of the other windows of $\sigma'$. Applying a suitable locally piecewise affine map $h$ which fixes the complement of $Q$ and maps $\vecv_j$ to the midpoint of $\ls[\vecv_{j-1},\vecv_{j+1}]$ we reduce the number of vertices in $V$ while leaving all the other windows unchanged. (See Figure~\ref{reduce-sides}.)

%
% It just remains to prove that if
%   \[ \sigma' = T \setminus (V_1 \cup \dots \cup V_n) \]
% where each window is a triangle, then we can apply a finite sequence of piecewise affine maps to move each triangle $V_k$ to the corresponding triangle $T_k$ in the description of our standard set $\sigma$. Let $V_k$ have vertices $\veca$, $\vecb$ and $\vecc$. Let $Q$ be any convex quadrilateral which has $\vecb$ and $\vecc$ as vertices and $\veca$ as an interior point (see Figure~\ref{move-triangles}). If $Q$ does not touch any of the other triangular windows in $\sigma'$ then we may apply a suitable piecewise affine map which `moves' $\veca$ to any other point in $Q$ while fixing the other vertices of $V_k$ and all the vertices of the other triangles. Doing this repeatedly allows one to move the $n$ triangles in $\sigma'$ to any configuration one wishes inside $T$ while maintaining the isomorphism class of the corresponding function spaces. In particular we see that $\AC(\sigma') \simeq \AC(\tau)$, and this completes the proof.

It just remains to prove that if
  \[ \sigma' = T \setminus (V_1 \cup \dots \cup V_n) \]
where each window is a triangle, then we can apply a finite sequence of locally piecewise affine maps to move the triangles $\{V_k\}$ to the corresponding triangles $\{T_k\}$ in the description of our standard set $\tau$. Our main tool is Lemma~\ref{TinR} which allows us to move a triangle anywhere within the interior of rectangular region while leaving everything outside the rectangle undisturbed. 
Although in concrete examples it is easy to efficiently move the triangles to their final position, for completeness we shall now give a general algorithm shows that this is always possible. 

Note that it follows from Lemma~\ref{triangle-move} that one may always move a vertex of a triangle to any point in the interior of that triangle, or, by applying two such moves, shrink any triangle towards one of its vertices.

We shall use the lexicographical ordering of points in the plane to choose the smallest vertex $(x_k,y_k)$ for each of the triangles $V_k$. 

The steps in the algorithm are as follows.
\begin{enumerate}
 \item Label the triangles so that $x_1 \le x_2 \le \dots \le x_n$.
 \item Starting from the right, shrink as many triangles (toward one vertex say) as is necessary to ensure that the $x$-coordinates of the smallest vertex of each of the triangles are distinct.
 \item Starting from the left, shrink each triangle towards its smallest vertex. If the triangles are shrunk to a sufficiently small size then the projections of these triangles onto the $x$-axis will form disjoint intervals $[a_k,b_k]$. Indeed, after sufficient shrinking the triangles will sit within the interiors of disjoint rectangles $R_k$ as in Figure~\ref{algorithm}. 
%  \item For each triangle, one now use Lemma~\ref{TinR} to shift the triangle so that all the vertices lie above the $x$-axis.
 \item Using Lemma~\ref{TinR} use a sequence of locally piecewise affine maps to move the $k$th triangle (within rectangle $R_k$ to one with vertices at $(a_k,0)$, $(b_k,0)$ and $(a_k,e_k)$ where $e_k$ is chosen small enough so that this triangle sits in $R_k$
%  \item Now move the smallest vertex to $(a_k,0)$, the bottom right-most vertex to $(b_k,0)$ and the remaining vertex to a position above the midpoint of the interval.  (See Figure~\ref{algorithm2}.)
 \item It remains to move the triangles to the correct positions to form the standard configuration $\tau$. Choose $\delta < \min(a_1,\frac{1}{3n})$. Starting from the left, move each triangle in turn (using Lemma~\ref{TinR}) so that it has vertices $(\frac{\delta(3k-1)}{3n},0)$, $(\frac{\delta(3k-2)}{3n},0)$ and $(\frac{\delta(2k-1)}{2n},\frac{1}{3n})$. 
 \item Now starting from the right, one can move the $k$-th triangle to the standard triangle $T_k$.
\end{enumerate}
% 
% The steps in the algorithm are as follows.
% \begin{enumerate}
%  \item Move vertices of as many triangles as is necessary to ensure that the $x$-coordinates of the  incentres of each of the triangles is distinct.
%  \item Using three of the vertex move maps one may shrink any triangle to an arbitrarily small triangle with the same incentre. If the triangles are shrunk to a sufficiently small size then the projections of these triangles onto the $x$-axis will form disjoint intervals. 
% %Without loss we may assume that the length of each of these intervals is less than $\frac{1}{3n}$. 
% (See Figure~\ref{algorithm}.)
%  \item For each triangle, one now can move one vertex at a time to shift the triangle to a position where the base of the triangle is the corresponding interval in the $x$-axis and the remaining vertex is positioned above the midpoint of the interval. 
% Label these triangles now $V_1',\dots,V_n'$ ordered from left to right. 
% (See Figure~\ref{algorithm2}.)
%  \item At the this stage it is now an easy matter to scale and translate $V_1',\dots,V_n'$ to the positions of the standard triangles $T_1,\dots,T_n$.
% % Since the base of the triangles $V_k'$ is smaller than  $\frac{1}{3n}$, at least one of these triangles can be moved to the corresponding standard triangle $T_k$. One can therefore move the triangles one by one until the procedure is complete.
% \end{enumerate}
Since we have only applied a finite sequence of locally piecewise affine maps, $\AC(\sigma') \simeq \AC(\tau)$, and this completes the proof.

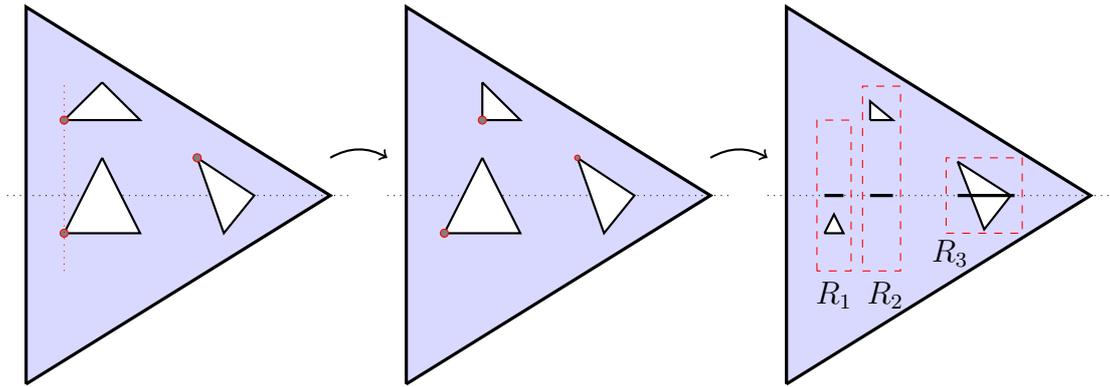
\begin{figure}[!ht]
\begin{center}
\begin{tikzpicture}[scale=0.5]
%
% Start
%
 \fill[blue!15] (-2,-5) -- (6,0) -- (-2,5) -- (-2,-5);
 \draw[very thick,black] (-2,-5) -- (6,0) -- (-2,5) -- (-2,-5);
 \draw[black,dotted] (-2.5,0) -- (6.5,0);

 \fill[white] (0,3) -- (-1,2) -- (1,2) -- (0,3);
 \draw[thick,black] (0,3) -- (-1,2) -- (1,2) -- (0,3);
 \draw[red] (-1,2) node[circle, draw, fill=black!50,inner sep=0pt, minimum width=3pt] {};
 \fill[white] (0,1) -- (-1,-1) -- (1,-1) -- (0,1);
 \draw[thick,black]  (0,1) -- (-1,-1) -- (1,-1) -- (0,1);
 \draw[red] (-1,-1) node[circle, draw, fill=black!50,inner sep=0pt, minimum width=3pt] {};
 \fill[white] (2.5,1) -- (4,0) -- (3.2,-1) -- (2.5,1);
 \draw[thick,black] (2.5,1) -- (4,0) -- (3.2,-1) -- (2.5,1);
 \draw[red] (2.5,1) node[circle, draw, fill=black!50,inner sep=0pt, minimum width=3pt] {};
 \draw[red,dotted] (-1,-2) -- (-1,3);

  \path[thick,->] (6,1) edge [bend left] (7.5,1);

 \fill[blue!15] (8,-5) -- (16,0) -- (8,5) -- (8,-5);
 \draw[very thick,black] (8,-5) -- (16,0) -- (8,5) -- (8,-5);
 \draw[black,dotted] (7.5,0) -- (16.5,0);

 \fill[white] (10,3) -- (10,2) -- (11,2) -- (10,3); 
 \draw[thick,black] (10,3) -- (10,2) -- (11,2) -- (10,3);
 \draw[red] (10,2) node[circle, draw, fill=black!50,inner sep=0pt, minimum width=3pt] {};
 \fill[white] (10,1) -- (9,-1) -- (11,-1) -- (10,1);
 \draw[thick,black] (10,1) -- (9,-1) -- (11,-1) -- (10,1); 
 \draw[red] (9,-1) node[circle, draw, fill=black!50,inner sep=0pt, minimum width=3pt] {};
 \fill[white] (12.5,1) -- (14,0) -- (13.2,-1) -- (12.5,1);
 \draw[thick,black] (12.5,1) -- (14,0) -- (13.2,-1) -- (12.5,1);
 \draw[red] (12.5,1) node[circle, draw, fill=black!50,inner sep=0pt, minimum width=2pt] {};

 \path[thick,->] (16,1) edge [bend left] (17.5,1);

 \fill[blue!15] (18,-5) -- (26,0) -- (18,5) -- (18,-5);
 \draw[very thick,black] (18,-5) -- (26,0) -- (18,5) -- (18,-5);
 \draw[black,dotted] (17.5,0) -- (26.5,0);

 \fill[white] (19,-1) -- (19.5,-1) -- (19.25,-0.5) -- (19,-1);
 \draw[thick,black] (19,-1) -- (19.5,-1) -- (19.25,-0.5) -- (19,-1);
%  \draw[red,dashed] (19,-1.3) -- (19,0);
%  \draw[red,dashed] (19.5,-1.3) -- (19.5,0);
 \draw[red,dashed] (18.8,-2) -- (18.8,2) -- (19.7,2) -- (19.7,-2) -- (18.8,-2);
 \draw[black] (19.25,-2) node[below]{$R_1$};
 \draw[black,very thick] (19,0) -- (19.5,0);

 \fill[white] (20.2,2) -- (20.8,2) -- (20.2,2.5) -- (20.2,2);
 \draw[thick,black] (20.2,2) -- (20.8,2) -- (20.2,2.5) -- (20.2,2);
%  \draw[red,dashed] (20,0) -- (20,2);
%  \draw[red,dashed] (21,0) -- (21,2);
 \draw[red,dashed] (20,-2) -- (21,-2) -- (21,2.9) -- (20,2.9) -- (20,-2);
 \draw[black] (20.6,-2) node[below]{$R_2$};
 \draw[black,very thick] (20.2,0) -- (20.8,0);

 \fill[white] (22.5,0.9) -- (23.9,0) -- (23.2,-0.9) -- (22.5,0.9);
 \draw[thick,black] (22.5,0.9) -- (23.9,0) -- (23.2,-0.9) -- (22.5,0.9);
%  \draw[red,dashed] (22.5,1) -- (22.5,-1);
%  \draw[red,dashed] (24,1) -- (24,-1);
 \draw[red,dashed] (22.2,-1) -- (24.2,-1) -- (24.2,1) -- (22.2,1) -- (22.2,-1);
  \draw[black] (22.3,-0.9) node[below]{$R_3$};
 \draw[black,very thick] (22.5,0) -- (24,0);

\end{tikzpicture}
\caption{Steps in the algorithm to map $\sigma'$ to $\tau$: (2)~making the smallest vertices distinct; (3)~shrinking the triangles so they have disjoint projections on the $x$-axis.}\label{algorithm}
\end{center}
\end{figure}

\end{proof}

\medskip\noindent
\textit{Acknowledgements.} The authors would like to thank Michael Cowling and Hanning Zhang for some helpful discussions regarding this work.

%%%%%%%%%%%%%%%%%%%%%%%%%%%%%%%%%%%%%%%%%%%%%%%%%%%%%%%%%%%%%%%%%%%%%%%
%
%   Bibliography
%
%%%%%%%%%%%%%%%%%%%%%%%%%%%%%%%%%%%%%%%%%%%%%%%%%%%%%%%%%%%%%%%%%%%%%%%
\bibliographystyle{amsalpha}

\end{document}